\documentclass[12pt, dvipsnames]{article}

    \author{Manuel M. M{\"u}ller$^\ast$, Yuetian Luo$^\dagger$ and Rina Foygel Barber$^\ddagger$\\ \\
$^\ast$\textit{Statistical Laboratory, University of Cambridge}\\
$^\dagger$\textit{Data Science Institute, University of Chicago}\\
$^\ddagger$\textit{Department of Statistics, University of Chicago}
}
    
    \title{Are all models wrong?\linebreak Fundamental limits in distribution-free empirical model falsification}
    
    \newenvironment{keywords}
    {\bgroup\leftskip 20pt\rightskip 20pt \small\noindent{\bfseries
    Keywords:} \ignorespaces}%
    {\par\egroup\vskip 0.25ex}

    \newcommand{\acks}[1]{\section*{Acknowledgments}#1}

        \usepackage{amsfonts, amsmath, amssymb, amsthm}
        \usepackage[margin = 1in]{geometry}

        \newcommand{\one}{\mathbbm{1}}
        
        \newtheorem{theorem}{Theorem}

        \newtheorem{lemma}[theorem]{Lemma}
        
        \newtheorem{corollary}[theorem]{Corollary}
        \newtheorem{definition}[theorem]{Definition}
        \newtheorem{proposition}[theorem]{Proposition}

        \usepackage{natbib}
        \usepackage{nameref}
     
        \RequirePackage[colorlinks,citecolor={blue!60!black},urlcolor={blue!70!black},linkcolor={red!60!black},breaklinks,hypertexnames=false]{hyperref}

\usepackage{xcolor}
\usepackage{scrextend}
\usepackage{tikz}
\usetikzlibrary{decorations.pathreplacing}
\usepackage[normalem]{ulem}
\usepackage{enumerate}

\usepackage{bbm, bm}

\usepackage{etoolbox}
\usepackage{comment}

\newcommand{\R}{\mathbb{R}}

\renewcommand{\P}{\mathbb{P}}
\newcommand{\E}{\mathbb{E}}

\newcommand{\Fcal}{\mathcal{F}}
\newcommand{\Dcal}{\mathcal{D}}
\newcommand{\Qcal}{\mathcal{Q}}

\newcommand{\Zcal}{\mathcal{Z}}

\newcommand{\bX}{\mathbf{X}}
\newcommand{\bx}{\mathbf{x}}

\newcommand{\bY}{\mathbf{Y}}
\newcommand{\bI}{\mathbf{I}}

\newcommand{\Unif}{\mathrm{Unif}}

\newcommand\Fpwc{\Fcal_{\mathrm{pwc}}}
\newcommand\Flin{\Fcal_{\mathrm{lin}}}

\newcommand{\sample}{\Dcal}

\newcommand{\dtv}{\mathrm{d}_{\mathrm{TV}}}

\newcommand{\RPF}{R_P(\Fcal)}
\newcommand{\RhatFn}{\hat{R}(\Fcal,\sample_n)}
\newcommand{\RhatFnB}{\hat{R}(\Fcal,\sample_n;B)}
\newcommand{\RhatFN}{\hat{R}(\Fcal,\sample_N)}
\newcommand{\iidsim}{\stackrel{\mathrm{iid}}{\sim}}

\NewExpandableDocumentCommand{\mytextsuperscript}{m}{\raisebox{1.25ex}{\scriptsize #1}}

\makeatletter
\newcommand{\labeltext}[2]{%
  \@bsphack
  \MakeLinkTarget*{#1}%
  \def\@currentlabel{#1}{\label{#2}}%
  \@esphack
}
\makeatother

\date{\today}

\begin{document}

\maketitle

\begin{abstract}
    In statistics and machine learning, when we train a fitted model on available data, we typically want to ensure that we are searching within a model class that contains at least one accurate model---that is, we would like to ensure an upper bound on the \emph{model class risk} (the lowest possible risk that can be attained by any model in the class). However, it is also of interest to establish lower bounds on the model class risk, for instance so that we can determine whether our fitted model is at least approximately optimal within the class, or, so that we can decide whether the model class is unsuitable for the particular task at hand. 
    Particularly in the setting of interpolation learning where machine learning models are trained to reach zero error on the training data, we might ask if, at the very least, a positive lower bound on the model class risk is possible---or are we unable to detect that ``all models are wrong''?
    In this work, we answer these questions in a distribution-free setting by establishing a model-agnostic, fundamental hardness result for the problem of constructing a lower bound on the best test error achievable over a model class, and examine its implications on specific model classes such as tree-based methods and linear regression.
\end{abstract}


\section{Introduction}\label{sec:intro}

A central goal in machine learning and statistics is to learn a rule that captures properties of interest of an unknown, data-generating distribution. For instance, in a regression setting, we might seek to identify a function whose output can accurately predict the response given an observed covariate. Typically in such a setting, we do not allow the fitted function to take any arbitrary form, but rather require it to be chosen from a (parametric or non-parametric) family of functions, a so-called \emph{model class}, $\Fcal$. Such restrictions to certain model classes can have many different motivations, such as encoding prior information, aiding interpretability, or allowing computationally efficient selection of a function. Indeed, any specific choice of analysis method restricts our model class in some way. This restriction may be explicit---for instance, with linear regression, we clearly limit ourselves to linear functions. But such a restriction can also be more implicit, such as through the choice of a specific neural network architecture, which is known to restrict the set of functions that can be represented, or the employed method used to eventually select a rule from the model class \citep{devore2021neural, petrova2023limitations}. 

Choosing a model class $\Fcal$ involves trading off between multiple considerations. A model class $\Fcal$ that is too constrained might mean that there is no good model $f\in\Fcal$ (all models in the class are too simple to capture the true distribution). On the other hand, a model class $\Fcal$ that is too complex may lead to challenges in selecting a good model $f\in\Fcal$, if the available sample size is too small. In statistical learning theory, this is often referred to as the \emph{approximation--estimation} tradeoff \citep{shalev2014understanding}. Moreover, an overly rich model class $\Fcal$ can also lead to computational constraints, where searching within $\Fcal$ is too costly, leading to an approximation--estimation--computation tradeoff \citep{bottou2007tradeoffs}. 

With this tradeoff in mind, we may ask the following question: after using our observed data to select a model $\hat{f}\in\Fcal$, can we determine whether this selected model $\hat{f}$ is, at least approximately, optimal within this model class $\Fcal$? This question has deep practical importance---if  we are dissatisfied with $\hat{f}$'s accuracy (say, for predicting future data), then if we determine that $\hat{f}$ is nearly optimal within $\Fcal$, we will need to consider switching to a richer class of models than $\Fcal$; on the other hand, if $\hat{f}$ is far from optimal within $\Fcal$, then we might simply need to gather more data in order to find a better model in the same class. 

To make this question more precise, let $R_P(f)$ denote the risk of a model relative to the (unknown) data distribution $P$---for example, the misclassification probability, in the context of a classification problem. Then we would like to determine whether the quantity $R_P(\hat{f}) - \inf_{f\in\Fcal}R_P(f)$ (sometimes referred to as the \emph{estimation error} \citep{devroye2013probabilistic}) is approximately zero, without relying on untestable assumptions on $P$. In order to do so, let us examine the feasibility of estimating both quantities: $R_P(\hat{f})$, the risk of our fitted model, and $\inf_{f\in\Fcal}R_P(f)$, which we will refer to as the \emph{model class risk}. Indeed, the difficulty of estimation is very different for these two quantities:
\begin{itemize}
    \item Estimating $R_P(\hat{f})$ has a simple solution: by training $\hat{f}$ on only part of the available data, we may simply use the remaining data as a holdout set to provide an estimate of $R_P(\hat{f})$.
    \item Estimating $\inf_{f\in\Fcal}R_P(f)$, on the other hand, is potentially quite challenging. Particularly in an overparameterized regime where a sample of size $n$ can be interpolated by the model class, minimizing the empirical risk does not necessarily provide a reliable estimate of the true model class risk $\inf_{f\in\Fcal}R_P(f)$.
\end{itemize}
While the former question is hence easy to address with a holdout set, in this paper we investigate the latter question. Is it possible to provide nontrivial bounds on the model class risk $\inf_{f\in\Fcal}R_P(f)$, without placing assumptions on $P$---that is, in a distribution-free setting?

\paragraph{Outline.} The remainder of this paper is organized as follows. 
Section~\ref{sec:background} defines the questions of interest in more detail, and provides background and definitions.
In Section~\ref{sec:extremes}, we present our results for two regimes, where the model class $\Fcal$ is ``low-complexity'' or ``high-complexity'', while Section~\ref{sec:boundary} examines examples lying in an interesting in-between regime. We conclude with a discussion in Section~\ref{sec:discussion}, including connections to related literature. Most proofs are deferred to the Appendix.

\section{Problem formulation and background}\label{sec:background}
In this section, we introduce some notation to formalize our questions, and give an overview
of our contributions on the problem of inference
 on the model class risk.
\subsection{Setting and notation}\label{sec:setting_and_notation}
Let $P$ denote the unknown distribution on the data space $\Zcal$, and let $\ell: \Fcal\times\Zcal\to[0,\infty)$ denote a loss function, where $\Fcal$ is the model class of interest. As a canonical example, in the setting of a binary classification problem, we might have $\Zcal = \R^d\times\{0,1\}$, where a data point $Z=(X,Y)$ consists of features $X\in\R^d$ and a binary response $Y\in\{0,1\}$, and we may then use the zero/one loss, which is defined as $\ell(f,Z) = \ell(f,(X,Y)) = \one\{f(X)\neq Y\}$. 

Given a choice of loss function $\ell$, we define the risk of a particular model $f$ as 
\[R_P(f) = \E_P[\ell(f,Z)],\]
\sloppy and define the model class risk as $\RPF:=\inf_{f\in\Fcal}R_P(f)$. Given a data set $\sample_n = (Z_1,\dots,Z_n) \in\Zcal^n$, the empirical risk is defined as
\[\hat{R}(f, \sample_n) = \frac{1}{n}\sum_{i=1}^n \ell(f,Z_i),\]
and we also write $\RhatFn := \inf_{f\in\Fcal}\hat{R}(f,\sample_n)$ to denote
the lowest risk achieved on the sample over all possible models $f\in\Fcal$.

In the regression setting with data points $Z_i=(X_i,Y_i)$, it is common to say that $\Fcal$ can \emph{interpolate} the observed data $\sample_n$ if we can find some $f\in\Fcal$ with $f(X_i)=Y_i$ for all $i\in[n]:=\{1,\dots,n\}$; in this work, we will use this term more broadly and will say that $\Fcal$ interpolates the data whenever 
$\RhatFn=0$, even though our work is not restricted to regression problems.

\subsection{Key questions: lower bounds and upper bounds}\label{sec:upper_lower}
As highlighted above, we are interested in performing distribution-free inference on the model class risk $\RPF$, which we formalize as follows. 
\begin{definition}\label{def:L_U_DF}
    Fix $\alpha \in (0,1)$, $n\geq 1$, and model class $\Fcal$. We say that $\hat{L}_\alpha(\Fcal, \cdot):  \Zcal^n \rightarrow [0, \infty]$ is a valid distribution-free lower bound on the model class risk of $\Fcal$ if 
    \begin{equation}\label{eqn:lowerbd_def:L_U_DF}
        \P_P\Bigl\{\RPF \geq \hat{L}_\alpha(\Fcal, \sample_n)\Bigr\} \geq 1 - \alpha\textnormal{\quad for all distributions $P$ on $\Zcal$,}
    \end{equation}
    and similarly, 
    $\hat{U}_\alpha(\Fcal, \cdot): \Zcal^n \rightarrow [0, \infty]$ is a valid distribution-free upper bound on the model class risk of $\Fcal$ if 
    \[
        \P_P\Bigl\{\RPF \leq \hat{U}_\alpha(\Fcal, \sample_n)\Bigr\} \geq 1 - \alpha\textnormal{\quad for all distributions $P$ on $\Zcal$,}\]
    where $\P_P\{\cdot\}$ denotes that the probability is computed with respect to $\sample_n\sim P^n$, i.e., a sample of size $n$ drawn i.i.d.\ from $P$.\footnote{We may also allow randomization in our lower and upper bounds. Formally, for the lower bound (and similarly for the upper bound), we define $\hat{L}_\alpha(\Fcal,\cdot,\cdot):\Zcal^n\times[0,1]\to[0,\infty]$, where $\hat{L}_\alpha(\Fcal,\sample_n,\xi)$ now depends also on a random seed $\xi$, and the probability in~\eqref{eqn:lowerbd_def:L_U_DF} is now computed with respect to both $\sample_n\sim P^n$ and $\xi\sim\textnormal{Unif}[0,1]$. All our results hold for both deterministic and randomized bounds, but for simplicity we suppress the random seed $\xi$ in our notation.}
\end{definition}

While the questions of obtaining lower and upper bounds appear symmetric in the above definition, in fact they are substantially different. This is because the target of inference is an infimum, $\RPF = \inf_{f\in\Fcal}R_P(f)$. Therefore, a valid upper bound on $\RPF$ can be obtained by simply bounding $R_P(f)$ for any \emph{single} choice of $f\in\Fcal$. For example, a common strategy would be to partition the available data $\sample_n$ into two independent subsets, $\sample_{n/2} = (Z_1,\dots,Z_{\lceil n/2\rceil})$ and $\sample_{n/2}' = (Z_{\lceil n/2\rceil+1},\dots,Z_n)$, then train a model $\hat{f}\in\Fcal$ using $\sample_{n/2}$, and finally use $\sample_{n/2}'$ as a holdout set to provide an unbiased estimate $\hat{R}(\hat{f}, \sample_{n/2}')$ of the risk $R_P(\hat{f})$ (and we can also construct an upper bound on $R_P(\hat{f})$, via concentration arguments). 

On the other hand, a lower bound on $\RPF$ is valid only if it bounds $R_P(f)$ simultaneously for \emph{all} $f\in\Fcal$, which is therefore a much more challenging task. The central aim of this paper is to explore the problem of constructing valid distribution-free lower bounds for $\RPF$. Of course, we may simply take $\hat{L}_\alpha(\Fcal,\cdot)\equiv 0$ to ensure validity, but this is not an informative lower bound. If we allow randomization in our answer, moreover, we can even provide a trivial solution without always returning zero: since a valid lower bound can have error $\alpha$ according to Definition~\ref{def:L_U_DF}, we may return a trivial lower bound 
\begin{equation}\label{eqn:trivial_alpha}\hat{L}_\alpha(\Fcal,\sample_n) = \begin{cases} 0, & \textnormal{ with probability $1-\alpha$},\\ \infty, & \textnormal{ with probability $\alpha$.}\end{cases}\end{equation}
In other words, any meaningful answer to this question must have $\P_P\big\{\hat{L}_\alpha(\Fcal,\sample_n)>0\} $ substantially larger than $\alpha$; we will therefore refer to any lower bound with $\P_P\big\{\hat{L}_\alpha(\Fcal,\sample_n)>0\} \leq \alpha + \mathrm{o}(1)$ as \emph{trivial}.
Thus, the central question of this work is to determine whether it is possible to construct a lower bound on $\RPF$ that is always valid, and is also nontrivial, to ensure that it is more informative than the meaningless solution constructed in~\eqref{eqn:trivial_alpha}.

\subsection{Overview of contributions}\label{sec:contribution}
We will now turn to the question of constructing valid distribution-free lower bounds for the model class risk.
As a special case, we can ask when it is possible to verify that $\RPF>0$, i.e., that there is no perfect model in the class. This aim recalls George E.~P.~Box's often-quoted claim that, in statistics, ``all models are wrong'' \citep{box1976science}. While this well-known quotation is referring to the idea that any practical choice of the model class $\Fcal$ must be wrong in terms of what it implies about the data distribution,\footnote{In other words, \cite{box1976science} uses the term ``model'' to refer to a class of functions $\Fcal$, such as ``the linear model'', while in this paper we use ``model'' to refer to a single $f\in\Fcal$.} here we interpret the question in a different way: given a fixed model class $\Fcal$, can we determine whether all models in the class are ``wrong'' in the sense that they do not lead to a zero risk?

As a starting point, consider the random variable
$\RhatFn = \inf_{f\in\Fcal}\hat{R}(f,\sample_n)$.
In expectation, this quantity is an underestimate of the target, since
\begin{equation}\label{eqn:lower_bound_E}\E[\RhatFn] = \E\left[\inf_{f\in\Fcal}\hat{R}(f,\sample_n)\right] \leq \inf_{f\in\Fcal}\E\left[\hat{R}(f,\sample_n)\right] = \inf_{f\in\Fcal}R_P(f) = \RPF.\end{equation}
Therefore we might expect that the empirical model class risk $\RhatFn$ could be a useful ingredient for constructing a distribution-free lower bound for $\RPF$. But if $\Fcal$ is a high-complexity model class and can interpolate the training data, we will have $\RhatFn=0$. We may believe that this is due solely to the high complexity of the model class $\Fcal$---but would it ever be possible to be confident that $\RPF>0$ in such a setting, or is $\hat{L}_\alpha(\Fcal,\sample_n)=0$ the only valid lower bound in this regime? More generally, is it ever possible to have a valid lower bound $\hat{L}_\alpha(\Fcal,\sample_n)$ that is higher than the empirical model class risk $\RhatFn$?

In this paper we find that the problem can be characterized by three regimes, which are defined via the \emph{interpolation capacity} of the model class $\Fcal$---the largest sample size $N$ for which $\RhatFN=0$. Of course, the definitions and results will be formalized below, but here we give an overview of our findings:
\begin{enumerate}[(i)]
    \item \textbf{The low-complexity regime.} If $\Fcal$ is not able to interpolate our data (i.e.,~$\RhatFn > 0$), then we can always construct a valid distribution-free lower bound $\hat{L}_\alpha(\Fcal,\sample_n)>0$.
    \item \textbf{The high-complexity regime.} At the other extreme, consider a highly complex model class $\Fcal$, which is able to interpolate not just our $n$ training points but is in fact able to interpolate a data set of\footnote{For $a, b \geq 0$, we write $a \gg b$ in informal descriptions of our results whenever we mean that $b/a = \mathrm{o}(1)$.} size $\gg n^2$. Then any valid distribution-free lower bound must necessarily be trivial, with $\P\{\hat{L}_\alpha(\Fcal,\sample_n)>0\}\leq \alpha + \mathrm{o}(1)$.
    \item \textbf{The in-between regime.} Finally, we find an interesting regime in between these two extremes. If $\Fcal$ interpolates a data set of size $n$, but its interpolation capacity is $\mathcal{O}(n^2)$, then the question does not have a general answer: we will see specific examples of model classes with qualitatively different behavior within this regime. Counterintuitively, it is sometimes possible to verify that $\RPF > 0$ (``all models are wrong'') even when $\RhatFn=0$ (i.e., the empirical risk does not provide any evidence that ``all models are wrong'').
\end{enumerate}
These three regimes are illustrated in Figure~\ref{fig:tikz_intro}. 

\begin{figure}[t]
    \centering
    \fbox{\begin{tikzpicture}

    \draw[thick, ->] (0,0) -- (12, 0);
    \draw[thick] (0, -0.15) -- (0, 0.15) node[anchor=north, yshift = -10.8] {$0$};
    \draw[thick] (4, -0.15) -- (4, 0.15) node[anchor=north, yshift = -10.8] {$n$};
    \draw[thick] (8, -0.15) -- (8, 0.15) node[anchor=north, yshift = -7] {$n^2$};

    \node[align = center] at (2,.5) {Low-complexity regime};
    \node[align = center] at (6,.5) {In-between regime};
    \node[align = center] at (10,.5) {High-complexity regime};

    \draw[thick, ->, dotted] (2,-0.1) -- (2,-1.1);
    \node[align = center] at (2.2, -1.8) {A valid $\hat{L}_\alpha(\Fcal,\Dcal_n)>0$\\always exists (see Corollary~\ref{cor:low_complexity_interpolation_capacity})};

    \draw[thick, ->, dotted] (6,-0.1) -- (6,-2.7);
    \node[align = center] at (6, -3.3) {Existence of a valid and nontrivial $\hat{L}_\alpha(\Fcal,\Dcal_n)$ \\ depends on $\Fcal$ (see Section~\ref{sec:boundary})};
    
    \draw[thick, ->, dotted] (10,-0.1) -- (10,-0.85);
    \node[align = center] at (10, -1.75) {Any valid $\hat{L}_\alpha(\Fcal,\Dcal_n)$ must be\\ trivial, with $\P(\hat{L}_\alpha(\Fcal,\Dcal_n)>0)$\\ $\leq \alpha + \mathrm{o}(1)$ (see Corollary~\ref{cor:high_complexity_interpolation_capacity})};

    \node[align = center] at (-1.5, 0) {Interpolation\\capacity of $\Fcal$};

    \node[align = center] at (-1.6, -2.5) {Our results};
    \draw [decorate,decoration={brace,amplitude=10pt}, thick, rotate=90](-3.9, 0.2) -- (-1.1, 0.2);

\end{tikzpicture}}
    \caption{A schematic depiction of the role of the complexity of $\Fcal$ (see Section~\ref{sec:contribution} for discussion).}
    \label{fig:tikz_intro}
\end{figure}
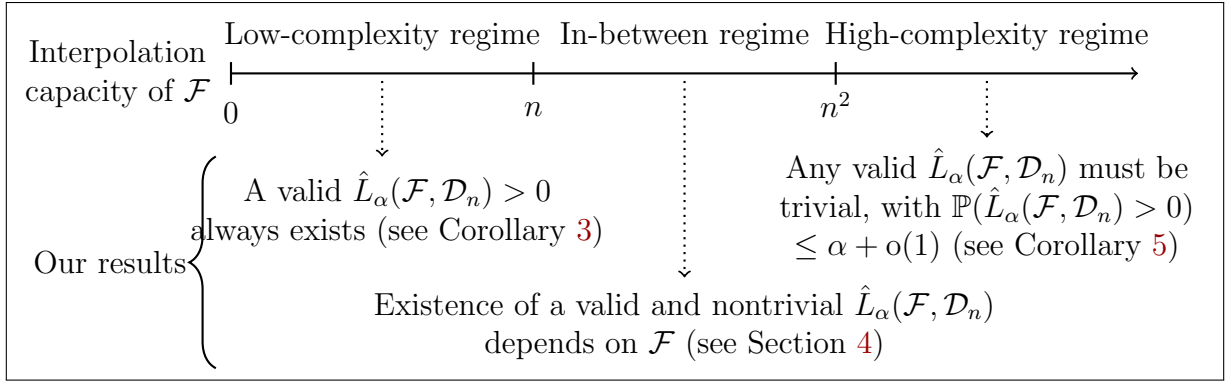

\section{Information-theoretic extremes: the low-complexity and high-complexity regimes}\label{sec:extremes}
In this section, we will develop theoretical results 
to determine whether it is possible to construct a meaningful and valid
lower bound on $\RPF$.

\subsection{Defining interpolation capacity}
As summarized in Figure~\ref{fig:tikz_intro}, the low-complexity
and high-complexity regimes are characterized by the \emph{interpolation capacity}
of the model class $\Fcal$---informally, how large of a sample can be interpolated
by some $f\in\Fcal$. Before presenting our theoretical
results, we begin by formalizing this measure of the complexity of $\Fcal$. These definitions are closely related to complexity measures such as VC dimension and fat-shattering dimension that appear in the statistical learning theory literature \citep{devroye2013probabilistic}.

Given a model class $\Fcal$ and a distribution $P$ on $\Zcal$ (and some fixed
choice of the
loss function $\ell$), we define 
\[N(\Fcal,P) = \sup\left\{n: \RhatFn=0\textnormal{ $P$-almost surely}\right\}.\] 
It will also be useful to define a related quantity, 
\[N_+(\Fcal,P) = \sup\bigg\{ n: \P_P\big\{\RhatFn = 0\big\} >0\bigg\}.\]
In other words, for a sample size $n$, if $n\leq N(\Fcal,P)$ then 
we will \emph{always} observe zero empirical risk, $\RhatFn=0$,
while if $n\leq N_+(\Fcal,P)$ then we \emph{may} observe zero empirical risk.

By definition, we have $N(\Fcal,P)\leq N_+(\Fcal,P)$---and in fact, in many 
common examples, these two measures of interpolation capacity are equal
or approximately equal. For instance,
in a regression setting with data points $(X_i,Y_i)\in\R^d\times\R=:\Zcal$,
    if $\Fcal$ is the class of all linear models (without an intercept), then $N(\Fcal,P)=N_+(\Fcal,P)=d$ for any distribution $P$ with a density on $\R^d\times\R$. 
For this reason, when discussing our results, we treat these two different measures of interpolation capacity
as essentially interchangeable, for the purpose of intuition.\footnote{There are however also cases in which $N(\Fcal, P) \ll N_+(\Fcal,P)$---for instance, consider the classification setting from Section~\ref{sec:setting_and_notation} with $\Fcal$ containing all constant functions on $\R^d$ and $P$ being a distribution of $(X, Y)$ on $\R^d \times \{0,1\}$ such that $\P_P(Y = 1|X)$ is in $(0, 1)$ almost surely. Then, $N(\Fcal, P) = 1$, while $N_+(\Fcal, P) = \infty$.}

\subsection{A lower bound via the empirical risk for the low-complexity regime}
We will now see that the empirical risk can 
provide a distribution-free lower bound on the model class risk $\RPF$.
\begin{theorem}\label{thm:low_complexity}
Fix $\alpha\in(0,1)$, $n\geq1$, and model class $\Fcal$. 
Then
\[\hat{L}_\alpha^{\mathrm{ERM}}(\Fcal,\sample_n) := \alpha\cdot \RhatFn\]
is a valid distribution-free lower bound on $\RPF$. 
\end{theorem}
\begin{proof}
If $\RPF=0$, then we must have $\RhatFn=0$ almost surely. If instead $\RPF>0$, then
by Markov's inequality, we have
$\P_P\{\hat{L}_\alpha^{\mathrm{ERM}}(\Fcal,\sample_n)\leq \RPF\}=\P_P\{\RhatFn\leq \alpha^{-1}\RPF\}\geq 1-\alpha$, since $\E_P[\RhatFn]\leq \RPF$ as computed in~\eqref{eqn:lower_bound_E}. 
\end{proof}
In general, this lower bound is far from optimal---if $n$ is large (and so the minimum empirical risk $\RhatFn$ exhibits strong concentration), we might hope for a lower bound that is approximately equal to $\RhatFn$, but $\hat{L}_\alpha^{\mathrm{ERM}}(\Fcal,\cdot)$ is a constant factor smaller. In the Appendix, we will establish a tighter result: in the setting of a bounded loss, we will construct a valid lower bound of the form $\hat{L}_\alpha(\Fcal,\sample_n) = (1-\mathrm{o}(1))\cdot\RhatFn$. 

However, even though the bound constructed here is loose, it is sufficient for answering our questions about whether a nontrivial lower bound is possible: to examine the implications of this theorem, 
the following corollary interprets this result in terms of the interpolation capacity.
\begin{corollary}\label{cor:low_complexity_interpolation_capacity}
In the setting of Theorem~\ref{thm:low_complexity}, suppose $N_+(\Fcal,P) < n$.
Then
\[\P_P\left\{\hat{L}_\alpha^{\mathrm{ERM}}(\Fcal,\sample_n)>0\right\} = 1.\]
\end{corollary}
\begin{proof}
If $N_+(\Fcal,P) < n$, then almost surely we have $\RhatFn>0$ and so $\hat{L}_\alpha^{\mathrm{ERM}}(\Fcal,\sample_n)>0$.
\end{proof}
That is, whenever $\Fcal$ fails to interpolate the training data, we are able
to conclude that ``all models are wrong'', as in the low-complexity regime
of Figure~\ref{fig:tikz_intro}. Outside of the low-complexity regime, on the other hand, the result of Theorem~\ref{thm:low_complexity} yields a lower bound $\hat{L}_\alpha^{\mathrm{ERM}}(\Fcal,\sample_n)=0$---that is, the lower bound on the model class risk is valid,
but meaningless. However, these results do not yet answer whether it is impossible for \emph{any} valid lower bound to be positive---it only tells us that this particular lower bound $\hat{L}_\alpha^{\mathrm{ERM}}(\Fcal,\cdot)$, computed via the empirical risk, is not informative. To resolve this remaining question, we will need to study the high-complexity regime.

\subsection{Hardness results in the high-complexity regime}

We now turn to the high-complexity regime, where we will see that there are fundamental
limits on the ability of \emph{any} distribution-free lower bound to provide
meaningful inference for $\RPF$. Based on the results above for the low-complexity setting, we might conjecture the following:
\begin{quote}
    Is it true that, if $\RhatFn=0$, then we must 
    have $\hat{L}_\alpha(\Fcal,\sample_n)=0$ as well (at least, with large probability) for any valid lower bound?
\end{quote}
That is, if $\Fcal$ can interpolate the training sample $\sample_n$ of size $n$,
is it \emph{impossible} to achieve a nontrivial and valid lower bound?
In fact, as previewed earlier in Figure~\ref{fig:tikz_intro}, this is not exactly the case, but a weaker result holds---if $\Fcal$ can interpolate a sample of size $N\gg n^2$, then indeed any valid lower bound $\hat{L}_\alpha(\Fcal,\sample_n)$ must (usually) equal zero.

Before stating our main result, we need to introduce some additional notation.
We will now consider an infinite stream of data points, $Z_1,Z_2,\dots\iidsim P$. Then $\sample_n=(Z_1,\dots,Z_n)$ is
a data set of $n$ i.i.d.\ draws from $P$, as before, but now we will also
write $\sample_N=(Z_1,\dots,Z_N)$ for each $N\geq n$, to denote larger
data sets that contain $\sample_n$ as a subset.
\begin{theorem}\label{thm:high_complexity}
Fix $\alpha\in(0,1)$, $n\geq1$, and model class $\Fcal$. 
Let $\hat{L}_\alpha(\Fcal, \cdot)$ be a valid distribution-free lower bound on the model class risk. Then, for all $N\geq n$,
\[
        \P_P\Bigl\{\hat{L}_\alpha(\Fcal, \sample_n)>  \RhatFN\Bigr\} \leq  \alpha + \frac{n^2}{2N}.
\]
\end{theorem}

The proof of this result follows the ``sample--resample'' strategy used for establishing certain hardness results in the distribution-free inference literature (see \citet{angelopoulos2024theoretical}), which relies on the fact that, when sampling $n$ times from a population of size $N$, 
the total variation distance between sampling with replacement and sampling without replacement is bounded by $\frac{n^2}{2N}$.

To interpret this theorem, let us return to the question of whether we 
can determine that ``all models are wrong'': can a valid distribution-free lower bound satisfy $\hat{L}_\alpha(\Fcal,\sample_n)>0$, certifying that the model class risk $\RPF$ is not zero, more often than the trivial solution in~\eqref{eqn:trivial_alpha}?
The following result shows how the interpolation capacity of $\Fcal$ allows us to characterize 
a regime where any valid lower bound must necessarily be trivial.
\begin{corollary}\label{cor:high_complexity_interpolation_capacity}
In the setting of Theorem~\ref{thm:high_complexity}, for any valid distribution-free lower bound on the model class risk, it holds that
\[\P_P\Bigl\{\hat{L}_\alpha(\Fcal, \sample_n) > 0\Bigr\} \leq \alpha + \frac{n^2}{2N(\Fcal,P)}.\]
\end{corollary}
\begin{proof}
\sloppy This result follows immediately from Theorem~\ref{thm:high_complexity}: for any $N\leq N(\Fcal,P)$, we
have $\RhatFN=0$ almost surely, and so $\P_P\{\hat{L}_\alpha(\Fcal,\sample_n)>0\} = \P_P\{\hat{L}_\alpha(\Fcal,\sample_n)> \RhatFN\}\leq \alpha + \frac{n^2}{2N}$.
\end{proof}

To summarize what we have seen for the low-complexity and high-complexity regimes in terms of the interpolation capacity, if $\Fcal$ can interpolate data sets of size $N\gg n^2$, then $\P_P\{\hat{L}_\alpha(\Fcal, \sample_n) > 0\} \leq \alpha + \mathrm{o}(1)$ for every valid lower bound $\hat{L}_\alpha(\Fcal,\cdot)$. Recalling the trivial solution in~\eqref{eqn:trivial_alpha}, this means that there do not exist any nontrivial valid lower bounds.
On the other hand, if $\Fcal$ cannot interpolate data sets of size $n$,
then there always exists a positive valid lower bound, namely, $\hat{L}_\alpha^{\mathrm{ERM}}(\Fcal,\cdot)$.

\section{Examples at the boundaries: the in-between regime}\label{sec:boundary}
Through the results of Section~\ref{sec:extremes} (as summarized in Figure~\ref{fig:tikz_intro}), we have seen that providing a nontrivial lower bound on the model class risk, using a sample of size $n$, is always possible if $\Fcal$ does not interpolate $n$ data points, but inevitably impossible if it interpolates $\gg n^2$ many data points. 

In this section, we shed some light on what happens in the ``in-between'' regime, where the interpolation capacity $N(\Fcal, P)$ lies somewhere between $n$ and $\mathcal{O}(n^2)$. Throughout this section, we will work in a regression setting, where we observe data points $(X_i,Y_i)\in\R^d\times\R=:\Zcal$, the model class $\Fcal$ is some set of functions $f:\R^d\to\R$, and $\ell$ is the squared loss, $\ell(f,(x,y))=(f(x)-y)^2$.

We will present two examples that exhibit different behavior in terms of how interpolation capacity relates to the problem of inference on $\RPF$:
\begin{itemize}
    \item \sloppy For one example (piecewise constant functions), we will see that a nontrivial lower bound is possible whenever $N(\Fcal,P) =\mathcal{O}(n^2)$---and in particular, we may have $N(\Fcal,P)\gg n$.
    \item For the second example (linear functions), we will see that it is impossible to construct a nontrivial lower bound whenever $N(\Fcal,P) \gg n$.
\end{itemize}
In particular, the contrast between these two examples implies that there is no universal phase transition for the problem of constructing a nontrivial lower bound: while a nontrivial lower bound is possible for any $\Fcal$ in the low-complexity regime, and impossible for any $\Fcal$ in the high-complexity regime, its existence in the in-between regime depends on the nature of $\Fcal$.

\subsection{Example: piecewise constant regression}\label{sec:piecewise_constant}
Define the model class of \emph{piecewise constant functions} with $\leq m$ components as
\[\Fpwc^{(m)} = \left\{ f:\R^d\to\R \  : \ \big|\{f(x):x\in\R^d\}\big|\leq m\right\}.\]
Equivalently, a function $f$ lies in $\Fpwc^{(m)}$ if it can be expressed as
\[f(x) = \sum_{i=1}^m y_i \cdot\one_{x\in A_i},\]
for some $y_1,\dots,y_m\in\R$ and some partition $\R^d=A_1\cup\dots\cup A_m$.
For example, methods such as regression trees or random forests will produce fitted models of this form.

For any distribution $P$ on $\R^d\times\R$ such that its marginal $P_X$ is nonatomic (i.e., $\P_P\{X=x\}=0$ for all $x\in\R^d$), we can calculate the interpolation complexity of this model class as
\[N_+(\Fpwc^{(m)},P) \geq N(\Fpwc^{(m)},P) \geq m,\]
since, for any $(X_1,Y_1),\dots,(X_m,Y_m)$ with distinct values $X_1,\dots,X_m$, we can construct a function $f\in\Fpwc^{(m)}$ with $f(X_i)=Y_i$ for each $i\in[m]$.
In particular, the lower bound 
\[\hat{L}_\alpha^{\mathrm{ERM}}(\Fpwc^{(m)},\sample_n) = \alpha\cdot \hat{R}(\Fpwc^{(m)},\sample_n)\]
constructed in Theorem~\ref{thm:low_complexity} is a valid but meaningless lower bound for any sample size $n\leq m$, since it is zero almost surely. However, we will now show that a better result can be obtained by constructing a different valid lower bound.

\begin{theorem}\label{thm:piecewise_constant_example}
    Fix $\alpha\in(0,1)$, and let $\alpha_0+\alpha_1=\alpha$, with $\alpha_0,\alpha_1>0$. Fix any $n\geq1$ and $m\geq 1$ with
    \[ m \leq \frac{n(n-1)}{2\log(1/\alpha_0)}.\] 
    Then
    \[\hat{L}_\alpha^{\mathrm{pwc}}(\Fpwc^{(m)},\sample_n) :=  
    \alpha_1 \cdot \hat{R}(\Fpwc^{(n-1)},\sample_n)\]
    is a valid distribution-free lower bound on $R_P(\Fpwc^{(m)})$.
\end{theorem}
The intuition for proving the above theorem is the following: consider any $f\in\Fpwc^{(m)}$, which takes (at most) $m$ distinct values in its output. 
If $m$ is not too large, then with high probability, at most $n-1$ of these values will actually be observed in a random sample of size $n$---that is, at least one observed value will be repeated. Consequently, on the sample $\sample_n$, the empirical risk of $f$ can also be achieved by some function in $\Fpwc^{(n-1)}$, which is a much less complex model class.

How should we interpret this result?
Consider a setting where $P_X$ and $P_Y$ are both nonatomic. In this case, $\hat{R}(\Fpwc^{(n-1)},\sample_n)>0$ must hold almost surely, since we cannot interpolate $n$ unique $Y$ values with any function $f\in\Fpwc^{(n-1)}$. In particular, since we can take $m\propto n^2$ in this theorem, this result
verifies that the fundamental hardness result provided in Theorem~\ref{thm:high_complexity} is tight: for this particular model class, although $N(\Fpwc^{(m)},P) \geq m\propto n^2$, it is nonetheless possible to provide a nontrivial valid lower bound, since $\hat{L}_\alpha^{\mathrm{pwc}}(\Fpwc^{(m)},\sample_n)>0$ almost surely.\footnote{We will also prove a stronger version of this result in the Appendix, that establishes a more precise connection between the scaling of $m$ and the behavior of a valid lower bound.}

\subsection{Example: linear models}\label{sec:linear}
For our second example, define the class of all linear predictors,
\begin{equation*}
	\Flin^{(d)} = \bigl\{x \mapsto x^\top \beta : \beta \in \R^d \bigr\}.
\end{equation*} For this class, for any distribution $P$ with a density on $\R^d\times \R$, we have
\[N_+(\Flin^{(d)},P) = N(\Flin^{(d)},P) = d,\]
since any $d$ data points in generic position can be interpolated by some $f\in\Flin^{(d)}$.
In particular, the valid lower bound $\hat{L}_\alpha^{\mathrm{ERM}}(\Flin^{(d)},\cdot)$, constructed in Theorem~\ref{thm:low_complexity}, provides a nontrivial lower bound for the model class risk if and only if $d<n$. However, might there be some other valid lower bound that is nontrivial even for $d\geq n$?

Our main result in this section will show that, in terms of finding a nontrivial lower bound, the performance of $\hat{L}_\alpha^{\mathrm{ERM}}(\Flin^{(d)},\cdot)$ is essentially the best that we could hope for, over a wide class of distributions.
That is, when $d\gg n$, it is impossible to construct a nontrivial lower bound.

We first need to set up some notation and definitions. To begin, let us define a class of noise-free linear distributions on $\R^d \times \R$,
\[\Qcal_{\textnormal{lin}} := \Bigl\{P \ : \textnormal{ for some $\beta\in\R^d$, $Y=X^\top\beta$ holds $P$-almost surely}\Bigr\},\]
and the class of $n$-fold product distributions from this class, \[\Qcal^{(n)}_{\textnormal{lin}} = \bigl\{P^n  : P\in\Qcal_{\textnormal{lin}}\bigr\}.\] 
Let
$\tilde{\Qcal}^{(n)}_{\textnormal{lin}}$ denote the family of mixtures of distributions in $\Qcal^{(n)}_{\textnormal{lin}}$, which can be seen as the convex hull of $\Qcal^{(n)}_{\textnormal{lin}}$.
Finally, for any distribution $P$ on $\R^d \times \R$, define
\[\lambda_{n,d}(P) := \inf_{  Q \in \tilde{\Qcal}^{(n)}_{\textnormal{lin}}} \dtv\bigl(P^n, Q  \bigr).\]

The following theorem provides a hardness result that bounds our ability to provide a nontrivial valid lower bound on the model class risk.
\begin{theorem} \label{thm:linear_example}
	Fix $\alpha \in (0,1)$, $n\geq 1$, and $d\geq 1$. Let $\hat{L}_\alpha(\Flin^{(d)},\cdot)$ be a valid distribution-free lower bound on the model class risk. Then for any distribution $P$ on $\R^{d} \times \R$,
    \[\P_P\left\{\hat{L}_\alpha(\Flin^{(d)},\sample_n) > 0 \right\} \leq \alpha  + \lambda_{n,d}(P).\]
\end{theorem}
For intuition, this holds simply due to the definition of total variation distance: if $\lambda_{n,d}(P)$ is low, then the distribution $P^n$ of the observed data is nearly indistinguishable from a convex combination of noise-free linear distributions---and consequently we cannot construct a nontrivial lower bound.

From this result, we can therefore see that any valid lower bound is essentially trivial whenever $\lambda_{n,d}(P)\approx 0$---but when will this be the case?  The next result establishes that, in a high-dimensional setting where $d\geq n$, the quantity $\lambda_{n,d}(P)$ is small for a broad class of distributions $P$.
\begin{proposition} \label{prop:general-bound-lambda}
Let $P$ be a distribution on $\R^d\times\R$ with a density. Let $(X_1,Y_1),\dots,(X_n,Y_n)\iidsim P$, and define $\bX\in\R^{n\times d}$ as the matrix with rows $X_i$.
If $d \geq n$, then for any positive definite matrix $\Omega\in\R^{d\times d}$,
	\[		 \lambda_{n,d}(P) \leq  \frac{1}{2}\E_P\left[\left|\frac{h_\Omega(\bX)}{\E_P[h_\Omega(\bX)]} - 1\right|\right], 
	\] where $h_\Omega(\bX) = (\det(\bX \Omega \bX^\top ))^{-1/2} $ (and we implicitly assume that $\E_P[h_\Omega(\bX)]<\infty$).
\end{proposition} 
Proposition \ref{prop:general-bound-lambda} suggests that as long as $h_\Omega(\bX)= (\det(\bX \Omega \bX^\top ))^{-1/2}$ concentrates for some positive definite matrix $\Omega$, we expect $\lambda_{n,d}(P)\approx 0$. To make this intuition more precise, consider the case $\Omega = \mathbf{I}_d$, and observe that we can write the determinant as
\[h_{\mathbf{I}_d}(\bX)^{-1} = \det(\bX\bX^\top)^{1/2} = \|X_1\|_2 \cdot \|\mathbf{P}_{X_1}^\perp X_2\|_2\cdot \|\mathbf{P}_{X_1,X_2}^\perp X_3\|_2 \cdot \hdots \cdot \|\mathbf{P}_{X_1,\dots,X_{n-1}}^\perp X_n\|_2,\]
where for each $i$, the notation $\mathbf{P}_{X_1,\dots,X_i}^\perp$ denotes projection to the orthogonal complement of the span of $X_1,\dots,X_i\in\R^d$. Thus, as long as we expect concentration of the norms of the $X_i$'s, and also expect that the $X_i$'s are not likely to be strongly collinear, we can expect $h_{\mathbf{I}_d}(\bX)$ to have low variance (and so $\lambda_{n,d}(P)$ will be small). 

To make this more concrete, the next result derives an explicit upper bound on $\lambda_{n,d}(P)$ for the special case of a Gaussian distribution.\footnote{We show in the Appendix how this type of bound can be extended to more general distributions.}
\begin{corollary}\label{cor:Gaussian-error}
    Let $P$ be any distribution on $\R^d\times\R$ with a density, such that the marginal distribution of $X$ is $P_X = \mathcal{N}(0,\Sigma)$ for some positive definite $\Sigma\in\R^{d\times d}$. Then, if $d\geq n+2$,
    \[
    \lambda_{n,d}(P) \leq \frac{1}{2}\sqrt{\frac{n}{d-n-1}}.
    \]
\end{corollary}
In particular, we have $\lambda_{n,d}(P)\approx 0$ for any $d \gg n$ in the Gaussian setting, meaning that any valid lower bound on the model class risk of $\Flin^{(d)}$ is essentially trivial---in particular, this holds for, say, $d\propto n^{1+a}$ for any $a>0$, meaning that we may still have $d\ll n^2$. 

This result therefore offers another example for the ``in-between'' regime of Figure \ref{fig:tikz_intro}, which complements the findings of Section~\ref{sec:piecewise_constant} for our first example: it verifies that the low-complexity regime results of Theorem~\ref{thm:low_complexity} and Corollary~\ref{cor:low_complexity_interpolation_capacity} are essentially tight, in the sense that if $N(\Fcal,P) \gg n$, then we cannot provide a universal guarantee that a nontrivial lower bound is always possible for any $\Fcal$. A related question is also addressed by \citet[Theorem~3]{kong2019estimating}, who show that even when restricted to the case where $P$ is known to be Gaussian but nothing is to be assumed about the covariance matrix, it is impossible to distinguish between $R_P(\mathcal{F}_{\mathrm{lin}}^{(d)}) > 0$ and $R_P(\mathcal{F}_{\mathrm{lin}}^{(d)}) = 0$ when given access to only $n = \mathcal{O}(d)$ data points.

\section{Discussion}\label{sec:discussion}
In this paper, we have discussed a central question in learning theory: when can we show that a model class is ``wrong'', in the sense that even the best instance of a model in the class cannot achieve zero test error? We do so by considering the more general task of lower bounding the model class risk empirically and in a distribution-free manner. Focusing on the case of highly flexible model classes capable of often perfectly interpolating the data, we identify a regime of ``mild'' interpolation (where, although $\Fcal$ interpolates the data $\sample_n$, it is not able to interpolate $\gg n^2$ many data points), a setting in which we may still empirically be able to perform nontrivial inference on $\RPF$. In contrast, in the high-complexity setting (where $\Fcal$ can interpolate $\gg n^2$ many data points), we may say that we are in a regime of ``hyper-interpolation'', where nontrivial inference is no longer possible. Our examples show how these regimes are sharply characterized.

\subsection{Related literature}
We will next discuss connections
to related areas in the statistics and machine learning literature.

\paragraph{Learning theory and the approximation--estimation tradeoff.}
As discussed in Section~\ref{sec:intro}, the model class risk is closely related to the question of excess risk $R_P(\hat{f}) - \inf_{f\in\Fcal}R_P(f)$ and the approximation--estimation tradeoff (see, for example,~\cite{bartlett2002model, bottou2007tradeoffs}). Specifically, the problem of finding an upper bound on excess risk is essentially equivalent in difficulty to the problem of a lower bound for $\RPF$---and consequently, the hardness results established in our work can be interpreted as evidence that a distribution-free upper bound on excess risk is also challenging in a regime where the model class $\Fcal$ is highly complex.

On the other hand, there are many exciting recent results on controlling the excess risk under only mild distributional assumptions. For instance, parametric convergence rates of the excess risk $R_P(\hat{f}) - \inf_{f \in \Fcal} R_P(f)$ can be achieved when the loss satisfies certain regularity conditions and the algorithm used to select $\hat{f}$ from $\Fcal$ satisfies stability properties \citep{klochkov2021stability}. \cite{mourtada2021improper} leveraged the idea of stability of the empirical risk minimizer and came up with a technique to construct models $\hat{f}$ (which are potentially improper, i.e., may take values not contained in $\Fcal$) achieving low excess risk in logistic regression and conditional density estimation problems. These results offer a sort of converse to the hardness results established in this paper, by examining settings where nontrivial guarantees can be achieved.

\paragraph{Overparametrization and interpolation.} 
In recent years, an exciting topic in machine learning research has been the \emph{benign overfitting} phenomenon, sometimes also called \emph{double descent}, where an increase in a model class's capacity beyond the interpolation threshold can help to reduce generalization error of the eventually fitted function \citep{belkin2019reconciling}. In other words, fitting a model $\hat{f}$ with $\hat{R}(\hat{f},\sample_n)=0$ (meaning, implicitly, that we are searching within in a class $\Fcal$ with $N(\Fcal,P)\geq n$) can often lead to high accuracy, in contrast to what earlier results, based on generalization bounds from the learning theory literature, might suggest. There are a number of analyses for the generalization error upper bound for different overfitting estimators; see for example \citet{belkin2018overfitting, belkin2019optimality, bartlett2020benign}. A few recent works have also provided lower bounds on the excess risk for estimators with large training errors to illustrate that interpolation or memorization is necessary for optimal generalization, for example, see \cite{cheng2022memorize} and references therein.

\paragraph{Distribution-free risk control.}
As discussed earlier, 
the problem of a distribution-free upper bound $\hat{U}_\alpha(\Fcal,\cdot)$ on the model class risk is fundamentally simpler than that of a lower bound. Concretely, as described in Section~\ref{sec:upper_lower}, we might train a model $\hat{f}\in\Fcal$ using $\sample_{n/2}$ (the first half of the available data set $\sample_n$), and then use the remaining data $\sample_{n/2}'$ to provide an upper bound on $\RPF$. If the loss function $\ell$ takes values in a bounded range $[0,B]$, then by Hoeffding's inequality~\citep{hoeffding1962probability}, with probability $\geq 1-\alpha$,
\[\RPF \leq  R_P(\hat{f}) \leq \hat{R}\big(\hat{f}, \sample_{n/2}'\big) + B\sqrt{\frac{\log(1/\alpha)}{2\lfloor n/2\rfloor}}=:\hat{U}_{\alpha}(\Fcal,\sample_n),\]
thus providing a distribution-free upper bound in the sense of Definition~\ref{def:L_U_DF}.
Such high-probability upper confidence bounds are the starting point for the distribution-free risk control methodology of \cite{bates2021distribution}; see also \cite{angelopoulos2022learntestcalibratingpredictive,angelopoulos2023conformalriskcontrol} for related approaches.

An important existing result on distribution-free inference on a model class risk is given by \citet[Theorem~8.5]{devroye2013probabilistic}, which illustrates that in a distribution-free classification setting, the task of estimating the Bayes error leads to a constant minimax lower bound.

\paragraph{Hardness of distribution-free inference.} Our work also contributes to an important line of work on impossibility or hardness results in distribution-free inference. For example, the classical work \cite{bahadur1956nonexistence} showed that inference for the mean of a random variable is impossible without any assumption on the distribution. More recent distribution-free hardness results have been established for problems such as predictive inference with conditional coverage \citep{vovk2012conditional,lei2014distribution,barber2021limits}, and inference on the conditional mean or median \citep{barber2020distribution,medarametla2021distribution}. Hardness results for assumption-free inference for algorithmic properties such as stability and risk have been established in \cite{kim2023black,luo2024limits}. This last result, which studies the problem of providing distribution-free inference on $\E[R_P(\hat{f})]$ for a model $\hat{f}$ fitted by some black-box algorithm $\mathcal{A}$, is the closest to our work in terms of the questions being investigated, but is still substantially different: while \citet{luo2024limits} study the problem of evaluating the risk of a model fitted by a specific algorithm $\mathcal{A}$, here we instead ask about the infimum risk $\inf_{f\in\Fcal}R_P(f)$ regardless of whether it is possible to construct an algorithm $\mathcal{A}$ that will identify an (approximately) optimal model.

\subsection{Extensions and open questions}
To conclude, we will briefly mention some possible directions for extensions and further research suggested by this work.
First, the questions examined in this work focus on a single model class $\Fcal$, but in practice we may want to choose $\Fcal$ in an adaptive way, by examining the performance of various fitted models on the data. In particular, returning to the approximation--estimation tradeoff, one approach to find a balance between these two conflicting goals is known as \emph{structural risk minimization}. Here, the idea is to consider a nested sequence of model classes $\Fcal_1 \subseteq\Fcal_2\subseteq\dots$, and to choose $\Fcal = \Fcal_{j_n}$ based on a suitable increasing sequence $j_n$, $n \geq 1$, or to use complexity regularization \citep{bartlett2002model, vonluxburg2008statisticallearningtheorymodels}. For instance, in a regression setting, $\Fcal_j$ may be a sequence of increasingly more complex neural network architectures that can approximate ever more complicated functions \citep{barron1994approximation}. In this setting, our results might be useful for identifying scenarios in which it is possible, or impossible, to use empirical evidence to help navigate this tradeoff. 

Second, recall that in this work we have identified an ``in-between'' regime: as summarized in Figure~\ref{fig:tikz_intro}, if the interpolation capacity of $\Fcal$ lies between $n$ and $\mathcal{O}(n^2)$, then it may be possible or impossible to provide distribution-free inference on $\RPF$. Our two examples (in Section~\ref{sec:boundary}) show that these boundaries cannot be improved---no universal guarantee of nontrivial inference can be achieved outside of the low-complexity regime, and no universal hardness result can be shown outside of the high-complexity regime.  However, an important
open question remains: might there be an alternative notion of model class complexity (i.e., different from the interpolation capacity), with which it would be possible to identify a universal phase transition directly from the low-complexity to high-complexity regime?

\acks{R.F.B. was partially supported by the National Science Foundation via grant DMS-2023109, and by the Office of Naval Research via grant N00014-24-1-2544. We thank John Duchi for bringing helpful references to our attention.}

\bibliography{mybib}

\clearpage

\numberwithin{theorem}{section}
\numberwithin{equation}{section}
\appendix

\section{Proofs and extensions for results in Section~\ref{sec:extremes}}\label{app:additional_proofs}

\subsection{A tighter lower bound for the low-complexity regime} \label{app:a-tighter-bound}

In this section, we present a stronger version of Theorem~\ref{thm:low_complexity}, for the setting of a bounded loss.
\begin{theorem}\label{thm:low_complexity_bounded_loss}
Fix $\alpha\in(0,1)$, $n\geq1$, and model class $\Fcal$. Assume that the loss function $\ell$ is bounded, taking values in $[0,B]$ for some $B>0$.
Define $\Delta_n\in (0,1]$ as the unique solution of\footnote{Note that, since $\Delta\mapsto -\Delta-\log(1-\Delta)$ is strictly increasing on $[0,1)$, with $\lim_{\Delta\searrow 0}\{-\Delta-\log(1-\Delta)\} = 0$ and $\lim_{\Delta\nearrow 1}\{-\Delta-\log(1-\Delta)\} = +\infty$, this is well-defined for any value $\RhatFn\geq 0$. In particular if $\RhatFn=0$, then this equation is solved by $\Delta_n=1$.}
\begin{equation}\label{eqn:compute_Delta_n}-\Delta_n - \log(1-\Delta_n) = \frac{B\log(1/\alpha)}{n\RhatFn}.\end{equation}
Then
\[\hat{L}_\alpha^{\mathrm{ERM},B}(\Fcal,\sample_n) := (1-\Delta_n) \RhatFn\]
is a valid distribution-free lower bound on $\RPF$. 
\end{theorem}
As for the lower bound constructed in Theorem~\ref{thm:low_complexity}, we again have $\hat{L}_\alpha^{\mathrm{ERM},B}(\Fcal,\sample_n)>0$ if and only if $\RhatFn>0$ (because this event is equivalent to $\Delta_n<1$).
But the theorem can be interpreted in a more quantitative way:
since $-\Delta_n - \log(1-\Delta_n)\geq  \Delta_n^2/2$, by~\eqref{eqn:compute_Delta_n} we have
\begin{equation} \label{ineq:easier-lower-bound}
   \hat{L}_\alpha^{\mathrm{ERM},B}(\Fcal,\sample_n) \geq  \biggl(1 - \sqrt{\frac{2B\log(1/\alpha)}{n\RhatFn}}\biggr) \RhatFn = \RhatFn - \sqrt{\RhatFn \cdot \frac{2B\log(1/\alpha)}{n}}. 
\end{equation}
In other words, our lower bound is of the form $(1-\mathrm{o}(1))\cdot \RhatFn$. 

In fact, this type of result can be applied to the setting of an unbounded loss, as well, via a truncation step. For any $f\in\Fcal$, define its truncated empirical risk as $\hat{R}(f,\sample_n;B) = \frac{1}{n}\sum_{i=1}^n \min\{\ell(f,Z_i),B\}$, and define $\RhatFnB = \inf_{f \in \Fcal}\hat{R}(f,\sample_n;B)$.
\begin{theorem}\label{thm:low_complexity_trunc_loss}
Fix $\alpha\in(0,1)$, $n\geq1$, and model class $\Fcal$. Fix any $B>0$, and
define $\Delta_{n,B}\in (0,1]$ as the unique solution of
\[-\Delta_{n,B} - \log(1-\Delta_{n,B}) = \frac{B\log(1/\alpha)}{n\RhatFnB}.\]
Then
\[\hat{L}_\alpha^{\mathrm{ERM-trunc},B}(\Fcal,\sample_n) := (1-\Delta_{n,B}) \RhatFnB\]
is a valid distribution-free lower bound on $\RPF$. 
\end{theorem}
To help interpret this lower bound, we can observe that as in~\eqref{ineq:easier-lower-bound} above, it holds that
\begin{multline}\label{ineq:easier-lower-bound-trunc}\hat{L}_\alpha^{\mathrm{ERM-trunc},B}(\Fcal,\sample_n) \geq \RhatFnB - \sqrt{\RhatFnB \cdot \frac{2B\log(1/\alpha)}{n}}\\\geq \RhatFnB - \sqrt{\frac{2B^2\log(1/\alpha)}{n}},\end{multline}
where the last step holds simply because $\RhatFnB\leq B$ by construction.

The proofs of these results rely on the following lemma.
\bigskip

\begin{lemma}\label{lemma:empirical_multiplicative_chernoff}
    Let $Z_1,\ldots, Z_n$ be independent random variables taking values in $[0,1]$. Denote $S := \sum_{i=1}^nZ_i$ and $\mu := \E[S]$. Fix $\alpha \in (0,1)$ and let $\Delta \in (0,1]$ be the unique solution to
    \[
    -\Delta -\log(1-\Delta)=\log(1/\alpha)/S.
    \]
    Then
    \[
    \P\bigl((1-\Delta)S\leq \mu\bigr)\geq 1-\alpha.
    \]
\end{lemma}

\begin{proof}[Proof of Lemma~\ref{lemma:empirical_multiplicative_chernoff}]
First we consider the degenerate case where $\mu=0$, i.e., $Z_i=0$ almost surely for each $i$. Then, almost surely, $S=0$ and $\Delta = 1$, and the result holds trivially. From this point on, then, we assume $\mu>0$ (note that we might still have $S=0$ with positive probability).

    By a multiplicative Chernoff bound \citep[Theorem~2.3(b)]{mcdiarmid1998concentration}, we have
    \[
        \P\Bigl\{S \leq (1+\delta)\mu\Bigr\} \geq 1 - \exp\bigl(-h(\delta)\mu\bigr)
    \]
    for $\delta > 0$, where $h(\delta)=(1+\delta)\log(1+\delta)-\delta$. Since $h$ is continuous and strictly increasing on $[0,\infty)$, with $h(0)=0$ and $h(\delta) \rightarrow \infty$ as $\delta \rightarrow \infty$, there exists a unique $\delta_\mu > 0$ such that $h(\delta_{\mu})=\log(1/\alpha)/\mu$, so that we have
    \[
        \P\Bigl\{S \leq (1+\delta_{\mu})\mu\Bigr\} \geq 1 - \alpha.
    \]
    To complete the proof, we therefore only need to verify that
     \[
        S \leq (1 + \delta_\mu)\mu  \ \Longrightarrow \ (1-\Delta)S \leq \mu .
     \]
    First, define
    \[\Delta' = \frac{\delta_\mu}{1+\delta_\mu} \in (0,1).\]
    Then $1+\delta_\mu = 1/(1-\Delta')$, and so
    \[S \leq (1 + \delta_\mu)\mu  \ \Longrightarrow \ (1-\Delta')S \leq \mu .\]
    Our last step is to check that, on the event $(1-\Delta')S \leq \mu$, we have $\Delta'\leq \Delta$. We have
    \[\frac{\log(1/\alpha)}{\mu} = h(\delta_\mu) = (1+\delta_\mu)\log(1+\delta_\mu) - \delta_\mu = \frac{1}{1-\Delta'}\log\left(\frac{1}{1-\Delta'}\right) - \frac{\Delta'}{1-\Delta'},\]
    so rearranging terms,
    \[-\Delta' - \log(1-\Delta') = (1-\Delta') \cdot \frac{\log(1/\alpha)}{\mu} \leq \frac{\log(1/\alpha)}{S} = -\Delta - \log(1-\Delta),\]
    where the inequality holds on the event that $(1-\Delta')S \leq \mu$, and 
    the last step holds by definition of $\Delta$. Finally, since $u\mapsto -u - \log(1-u)$ is a strictly increasing function on $[0,1)$, this implies $\Delta'\leq \Delta$, which completes the proof. 
\end{proof}

We are now in a position to state the proofs of Theorems~\ref{thm:low_complexity_bounded_loss} and~\ref{thm:low_complexity_trunc_loss}.\bigskip

\begin{proof}[Proof of Theorem~\ref{thm:low_complexity_bounded_loss}]
This result is simply a special case of Theorem~\ref{thm:low_complexity_trunc_loss} (since, when the loss $\ell$ takes values in $[0,B]$, we have $\RhatFnB=\RhatFn$).
\end{proof}
\begin{proof}[Proof of Theorem~\ref{thm:low_complexity_trunc_loss}]
For any $\varepsilon>0$, we can find some $f_\varepsilon\in\Fcal$ such that $\RPF\leq R_P(f_\varepsilon)\leq \RPF + \varepsilon$. Noting that $\frac{\hat{R}(f_\varepsilon, \sample_n;B)}{B}$ is an average of $n$ i.i.d.~terms, which each lie in $[0,1]$, we can apply Lemma~\ref{lemma:empirical_multiplicative_chernoff}, by taking $\Delta_\varepsilon$ to be the unique solution to 
    \[
    -\Delta_\varepsilon-\log(1-\Delta_\varepsilon)=\frac{B\log(1/\alpha)}{n\hat{R}(f_\varepsilon,\sample_n;B)}
    \]
    to see that
    \begin{multline*}
    		1-\alpha\leq \P\biggl\{(1 - \Delta_\varepsilon)\hat{R}(f_\varepsilon, \sample_n;B) \leq  \E[\hat{R}(f_\varepsilon, \sample_n;B)]\biggr\} \\\leq \P\biggl\{(1 - \Delta_\varepsilon)\RhatFnB \leq  \E[\hat{R}(f_\varepsilon, \sample_n;B)]\biggr\},
    \end{multline*}
    where the last step holds 
    since $\hat{R}(f_\varepsilon,\sample_n;B)\geq \RhatFnB$. And, again using the fact that $\hat{R}(f_\varepsilon,\sample_n;B)\geq \RhatFnB$, we have 
    \[
    -\Delta_\varepsilon-\log(1-\Delta_\varepsilon)=\frac{B\log(1/\alpha)}{n\hat{R}(f_\varepsilon,\sample_n;B)} \leq \frac{B\log(1/\alpha)}{n\RhatFnB}=-\Delta_{n,B}-\log(1-\Delta_{n,B})
    \]
    by definition of $\Delta_{n,B}$.
 Since $u \mapsto -u-\log(1-u)$ is strictly increasing on $[0,1)$, it follows that we must have $\Delta_\varepsilon \leq \Delta_{n,B}$. We conclude
    \begin{align*}
    		1 - \alpha \leq  \P\biggl\{(1 - \Delta_{n,B})\hat{R}(\Fcal, \sample_n;B) \leq  \E[\hat{R}(f_\varepsilon, \sample_n;B)]\biggr\}.
    \end{align*}
    Next, we also have
    \[\E[\hat{R}(f_\varepsilon, \sample_n;B)]\leq
    \E[\hat{R}(f_\varepsilon, \sample_n)] = R_P(f_\varepsilon) \leq \RPF + \varepsilon, \]
    where the first step holds since $\hat{R}(f,\sample_n;B)\leq \hat{R}(f,\sample_n)$ for any $f\in\Fcal$ by definition of the truncated empirical risk, and the last step holds by definition of $f_\varepsilon$. Therefore,
    \[\P\biggl\{(1 - \Delta_{n,B})\hat{R}(\Fcal, \sample_n;B) \leq  \RPF + \varepsilon\biggr\} \geq 1-\alpha,\]
    and since $\varepsilon > 0$ can be chosen to be arbitrarily small, the result follows.
\end{proof}

\subsection{Proof of Theorem~\ref{thm:high_complexity} (the hardness result)}\label{sec:proof_thm:high_complexity}
For any fixed values $z_1,\dots,z_N\in\Zcal$, let $Q=\frac{1}{N}\sum_{i=1}^N \delta_{z_i}$ denote their empirical distribution (where $\delta_z$ is the point mass at $z$). Letting $I_1,\dots,I_n\iidsim\Unif([N])$, we then see that $(z_{I_1},\dots,z_{I_n})$ is a data set that contains $n$ i.i.d.\ draws from $Q$. Since $\hat{L}_\alpha(\Fcal,\cdot)$ is a valid distribution-free lower bound, it must therefore be valid as a lower bound for $R_Q(\Fcal)$ when applied to data sampled from $Q$---and therefore,
\[\P\left\{\hat{L}_\alpha(\Fcal,(z_{I_1},\dots,z_{I_n})) \leq \hat{R}(\Fcal,(z_1,\dots,z_N))\right\}\geq 1-\alpha,\]
by noting that $R_Q(\Fcal) = \hat{R}(\Fcal,(z_1,\dots,z_N))$ by definition of $Q$ as an empirical distribution.
Next, consider indices $J_1,\dots,J_n\in[N]$ sampled uniformly \emph{without} replacement. By a total variation distance bound on the difference between sampling with and without replacement (e.g., \citet{stam1978distance}; see also \citet[Lemma 4.15]{angelopoulos2024theoretical}), we therefore have
\[\P\left\{\hat{L}_\alpha(\Fcal,(z_{J_1},\dots,z_{J_n})) \leq \hat{R}(\Fcal,(z_1,\dots,z_N))\right\}\geq 1-\alpha - \frac{n^2}{2N}.\]
Since this is true for any fixed sequence of values $z_1,\dots,z_N\in\Zcal$, it is also true for randomly sampled values, $Z_1,\dots,Z_N\iidsim P$---that is, we have
\[\P\left\{\hat{L}_\alpha(\Fcal,(Z_{J_1},\dots,Z_{J_n})) \leq \hat{R}(\Fcal,(Z_1,\dots,Z_N))\right\}\geq 1-\alpha -\frac{n^2}{2N},\]
where now the probability is computed with respect to both $Z_i$'s and $J_i$'s. But since the data points are i.i.d., by symmetry it is equivalent to write
\[\P\left\{\hat{L}_\alpha(\Fcal,(Z_1,\dots,Z_n)) \leq \hat{R}(\Fcal,(Z_1,\dots,Z_N))\right\}\geq 1-\alpha-\frac{n^2}{2N},\]
which proves the desired claim.

\section{Proofs and extensions for results in Section~\ref{sec:piecewise_constant}}

\subsection{Proof of Theorem~\ref{thm:piecewise_constant_example} (piecewise constant functions)}
Fix any $\varepsilon>0$, and choose some $f_\varepsilon\in\Fpwc^{(m)}$ with $R_P(f_\varepsilon) \leq \inf_{f\in\Fpwc^{(m)}}R_P(f) + \varepsilon = R_P(\Fpwc^{(m)}) + \varepsilon$. 
By definition of the model class, we can express $f_\varepsilon$ as
    \[f_\varepsilon(x) = \sum_{j=1}^m y_j\cdot\one_{x\in A_j},\]
for some $y_1,\dots,y_m\in \R$ and some partition $\R^d=A_1\cup\dots\cup A_m$.
Now define 
\[I(\sample_n) = \left\{j\in[m] : \sum_{i=1}^n\one_{X_i\in A_j} > 0\right\}\subseteq [m],\]
which indexes those sets $A_j$ that were observed at least once in the data set. Define a function
\[\hat{f}_\varepsilon(x) = \sum_{j\in I(\sample_n)} y_j\cdot\one_{x\in A_j} + y_{\min I(\sample_n)}\cdot\one\Big\{x\in \bigcup_{j\in[m]\backslash I(\sample_n)} A_j\Big\}.\]
We can observe that $\hat{f}_\varepsilon\in\Fpwc^{(|I(\sample_n)|)}$, since the function's output always lies in $\{y_j : j\in I(\sample_n)\}$, and moreover, $\hat{f}_\varepsilon(X_i) = f_\varepsilon(X_i)$ for all $i\in[n]$, by construction. 
On the event that $|I(\sample_n)|\leq n-1$,
we then have $\Fpwc^{(n-1)}\supseteq\Fpwc^{(|I(\sample_n)|)}\ni\hat{f}_\varepsilon$. Thus
\[\textnormal{If $|I(\sample_n)|\leq n-1$ \ then \, $\hat{R}(f_\varepsilon,\sample_n) = \hat{R}(\hat{f}_\varepsilon,\sample_n) \geq \hat{R}(\Fpwc^{(n-1)},\sample_n) = \alpha_1^{-1} \cdot \hat{L}_\alpha^{\mathrm{pwc}}(\Fpwc^{(m)},\sample_n)$.}\] 
Therefore, 
\begin{multline*}\P_P\left\{\hat{L}_\alpha^{\mathrm{pwc}}(\Fpwc^{(m)},\sample_n) \leq R_P(f_\varepsilon)\right\} \\\geq  \P_P\big\{\hat{R}(f_\varepsilon,\sample_n) \leq \alpha_1^{-1} R_P(f_\varepsilon)\textnormal{ and }|I(\sample_n)|\leq n-1\big\} \geq 1-\alpha_0-\alpha_1 = 1-\alpha,\end{multline*}
where the second step holds since $\hat{R}(f_\varepsilon,\sample_n) \leq \alpha_1^{-1} R_P(f_\varepsilon)$ with probability $\geq 1-\alpha_1$ by Markov's inequality, and $\P_P\{|I(\sample_n)|\leq n-1\}\geq 1-e^{-\frac{n(n-1)}{2m}}\geq 1- \alpha_0$  by Lemma~\ref{lem:occupancy} below together with our assumption on $m$. 
Recalling the definition of $f_\varepsilon$, we have therefore proved that
\[\P_P\left\{\hat{L}_\alpha^{\mathrm{pwc}}(\Fpwc^{(m)},\sample_n) \leq R_P(\Fpwc^{(m)}) + \varepsilon\right\} \geq 1-\alpha.\]
Since $\varepsilon>0$ can be taken to be arbitrarily small, this completes the proof of the theorem.

\begin{lemma}\label{lem:occupancy}
    \sloppy Consider a discrete distribution $P_Y$ with support $\{y_1,\dots,y_m\}$, and let $Y_1,\dots,Y_n\iidsim P_Y$. Let
    \[I = \sum_{j=1}^m \one\left\{\sum_{i=1}^n\one_{Y_i=y_j}>0\right\}\]
    denote the number of unique values observed in the sample. Then
    \[\P_{P_Y}\left\{I \leq n-1\right\} \geq 1-e^{-\frac{n(n-1)}{2m}}.\]
\end{lemma}

Proving bounds on this quantity $I$ is related to the \emph{occupancy problem} in probability theory (e.g., see \citet{ben2017concentration}); here the bound we need is slightly different from the type of results in the existing literature.
\begin{proof}[Proof of Lemma~\ref{lem:occupancy}]
If $m < n$, the result is immediate. Suppose from this point on then that $m \geq n$. Since $I$ takes values in the set $[n]$, we have $\P_{P_Y}\{I\leq n-1\} = 1 - \P_{P_Y}\{I=n\}$. In order to have $I=n$ we must have all observations $Y_1,\dots,Y_n$ distinct, meaning that
\[\P_{P_Y}\{I=n\}
    =\sum_{\textnormal{distinct }i_1,\dots,i_n\in[m]}\prod_{j=1}^n p_{i_j}
    =n! \cdot\sum_{\substack{S\subseteq[m]\\ |S|=n}}\prod_{i\in S}p_i = n! \cdot f(p),\]
    where $p_j := \P_{P_Y}(Y_1 = y_j)$, and where for a probability vector $p=(p_1,\dots,p_m)$ we define $f(p) = \sum_{\substack{S\subseteq[m]\\ |S|=n}}\prod_{i\in S}p_i $. The function $f$ attains its maximum at $p=(\frac{1}{m},\dots,\frac{1}{m})$ \citep{munford1977birthday}, and therefore,
\begin{multline*}\P_{P_Y}\{I=n\} = n!\cdot f(p) \\\leq n! \cdot f\left(\Big(\frac{1}{m},\dots,\frac{1}{m}\Big)\right) = n! \cdot {m\choose n} \left(\frac{1}{m}\right)^n = \frac{m(m-1)\hdots(m-n+1)}{m^n} \\= \left(1-\frac{1}{m}\right)\cdot\hdots\cdot\left(1-\frac{n-1}{m}\right)\leq e^{-1/m}\cdot \hdots\cdot e^{-(n-1)/m} = e^{-\frac{n(n-1)}{2m}}.\end{multline*}
\end{proof}

\subsection{Extensions for the piecewise constant example}
Next, we state and prove several extensions of Theorem~\ref{thm:piecewise_constant_example}, to give a more precise characterization of how the effective complexity $m$ of the model class $\Fpwc^{(m)}$ influences the behavior of the lower bound.
\begin{theorem}\label{thm:piecewise_constant_example_extension}
    Fix $\alpha\in(0,1)$, $n\geq1$, and $m\geq 1$. Let $\alpha_0+\alpha_1=\alpha$, with $\alpha_0,\alpha_1>0$.
    Then
    \[\hat{L}_\alpha^{\mathrm{pwc},r}(\Fpwc^{(m)},\sample_n) :=  
    \alpha_1 \cdot \hat{R}(\Fpwc^{(n-r)},\sample_n)\]
    is a valid distribution-free lower bound on $R_P(\Fpwc^{(m)})$, where
    \[r := \left\lceil \left(\frac{n(n-1)}{n+2m}  -  2\sqrt{\frac{n(n-1)}{n+2m} \log(1/\alpha_0)}\right)_+\right\rceil.\] 
\end{theorem}
To interpret this result, observe that if $P_X$ and $P_Y$ are nonatomic,
then we would typically expect to have 
\[\hat{R}(\Fpwc^{(n-r)},\sample_n) \propto \frac{r}{n},\]
for the squared loss. This is because the data contains $n$ i.i.d.\ $Y$ values, but the model class $\Fpwc^{(n-r)}$ only allows for functions $f$ that return $\leq n-r$ distinct $Y$ values, and so informally, we expect that $r/n$ of the variance of $Y$ remains unexplained by any $f\in\Fpwc^{(n-r)}$. 
For $m\geq n$, as long as $m\leq cn^2$ for some appropriately chosen constant $c$, we have $r\propto n^2/m$---and therefore, we expect that the lower bound will scale as
\[\hat{L}_\alpha^{\mathrm{pwc},r}(\Fpwc^{(m)},\sample_n) \propto \frac{n}{m}.\]
In particular, if $m= \mathcal{O}(n)$, then $\hat{L}_\alpha^{\mathrm{pwc},r}(\Fpwc^{(m)},\sample_n)$ is likely bounded away from zero---a very informative lower bound.

We are now ready to prove this extension. \bigskip

\begin{proof}[Proof of Theorem~\ref{thm:piecewise_constant_example_extension}]
Following an identical argument as in the proof of Theorem~\ref{thm:piecewise_constant_example}, and defining $f_\varepsilon$ as in that proof, we have
\begin{multline*}\P_P\left\{\hat{L}_\alpha^{\mathrm{pwc},r}(\Fpwc^{(m)},\sample_n) \leq R_P(f_\varepsilon)\right\} \\\geq  \P_P\big\{\hat{R}(f_\varepsilon,\sample_n) \leq \alpha_1^{-1} R_P(f_\varepsilon)\textnormal{ and }|I(\sample_n)|\leq n-r\big\} \geq 1-\alpha_0-\alpha_1 = 1-\alpha,\end{multline*}
as long as we can show that
\[\P_P\{|I(\sample_n)|\leq n-r\}\geq 1-\alpha_0,\]
which is established in Lemma~\ref{lem:occupancy_extension} below.
Therefore,
\[\P_P\left\{\hat{L}_\alpha^{\mathrm{pwc},r}(\Fpwc^{(m)},\sample_n) \leq R_P(\Fpwc^{(m)}) + \varepsilon\right\} \geq 1-\alpha,\]
and since $\varepsilon>0$ can be taken to be arbitrarily small, this completes the proof. 
\end{proof}

\begin{lemma}\label{lem:occupancy_extension}
    Consider a discrete distribution $P_Y$ with support $\{y_1,\dots,y_m\}$, and let $Y_1,\dots,Y_n\iidsim P_Y$. Let
    \[I = \sum_{j=1}^m \one\left\{\sum_{i=1}^n\one_{Y_i=y_j}>0\right\}\]
    denote the number of unique values observed in the sample. Then for any $\alpha_0\in(0,1)$,
    \[\P_{P_Y}\left\{I \leq n-r\right\} \geq 1-\alpha_0,\]
    where
        \[r := \left\lceil \left(\frac{n(n-1)}{n+2m}  -  2\sqrt{\frac{n(n-1)}{n+2m} \log(1/\alpha_0)}\right)_+\right\rceil.\] 
\end{lemma}
\begin{proof}[Proof of Lemma~\ref{lem:occupancy_extension}]
        First, for any $k\geq 1$ and any $Y=(Y_1,\dots,Y_k)\in\{y_1,\dots,y_m\}^k$, define 
    \[r(Y) = \sum_{j=1}^m (C_j(Y)-1)_+ ,\]
    where $C_j(Y) = \sum_{i=1}^n\one_{Y_i=y_j}$ counts the number of times that the value $y_j$ was observed in the vector $Y$.
    Fixing any $k\geq 2$, and defining $Y_{-i} = (Y_1,\dots,Y_{i-1},Y_{i+1},\dots,Y_k)\in \{y_1,\dots,y_m\}^{k-1}$, note that $C_j(Y)=C_j(Y_{-i})$ for all $i$ with $Y_i\neq y_j$. Therefore,
    for any $i,j$ with $Y_i=y_j$,
    \[r(Y) - r(Y_{-i}) = (C_j(Y)-1)_+ - (C_j(Y_{-i})-1)_+ = \one_{C_j(Y)\geq 2}.\]
    In particular, this implies 
    \[0 \leq r(Y) - r(Y_{-i}) \leq 1\]
    for all $i\in[k]$,
    and also,
    \[\sum_{i=1}^k \big(r(Y) - r(Y_{-i})\big) 
    = \sum_{j=1}^m C_j(Y) \cdot \one_{C_j(Y)\geq 2} \leq  \sum_{j=1}^m 2(C_j(Y)-1)_+ = 2r(Y).\]
    This means that $r$ is a $(2,0)$-strongly-self-bounding function, in the terminology of \citet[Section 6.11]{boucheron2013concentration}. Applying \citet[Theorem 6.20]{boucheron2013concentration}, then, for $Y=(Y_1,\dots,Y_n)$ where $Y_1,\dots,Y_n\iidsim P_Y$,
    \[\P_{P_Y}\{r(Y) \geq \E_{P_Y}[r(Y)] - t\} \geq 1- e^{-t^2/4\E_{P_Y}[r(Y)]} \]
    for all $t>0$, i.e., $r(Y)$ has a subgaussian left tail. Choosing $t = 2\sqrt{\E_{P_Y}[r(Y)]\log(1/\alpha_0)}$,
    \begin{equation}\label{eqn:self_bounding_alpha0}\P_{P_Y}\left\{r(Y) \geq \E_{P_Y}[r(Y)] - 2\sqrt{\E_{P_Y}[r(Y)]\log(1/\alpha_0)}\right\} \geq 1- \alpha_0. \end{equation} 
    
    Next we calculate $\E_{P_Y}[r(Y)]$.  First, $C_j\sim\mathrm{Binom}(n,p_j)$, where $p_j = \P_{P_Y}\{Y_1=y_j\}$. Thus
    \[\E_{P_Y}[(C_j-1)_+] = \E_{P_Y}[C_j - 1 + \one_{C_j=0}]
    = np_j - 1 + (1-p_j)^n,\]
    and so
    \[\E_{P_Y}[r(Y)] = \sum_{j=1}^m \big(np_j - 1 + (1-p_j)^n\big) = n - m + \sum_{j=1}^m (1-p_j)^n.\]
    Now let $p:=(p_1,\dots,p_m)$, which lies in the probability simplex. Since $p\mapsto \sum_{j=1}^m (1-p_j)^n$ is convex, this means that $\sum_{j=1}^m (1-p_j)^n$ is minimized at $p=(\frac{1}{m},\dots,\frac{1}{m})$, i.e.,
    \[\E_{P_Y}[r(Y)] \geq n - m + m\left(1-\frac{1}{m}\right)^n \geq \frac{n(n-1)}{n+2m},\]
    where the last step holds by \citet[Lemma 4]{lee2021distribution}. 

    Combining this lower bound on $\E_{P_Y}[r(Y)]$ with the
    calculation~\eqref{eqn:self_bounding_alpha0}, we therefore have
    \[\P_{P_Y}\left\{r(Y) \geq \left(\frac{n(n-1)}{n+2m} - 2\sqrt{\frac{n(n-1)}{n+2m}\cdot \log(1/\alpha_0)} \right)_{+} \right\} \geq 1- \alpha_0,\]
    since $t\mapsto (t - 2\sqrt{t\log(1/\alpha_0)})_+$ is a nondecreasing function on $t\geq 0$, and $r(Y)$ is nonnegative. By definition of $r$, and using the fact that $r(Y)$ is integer-valued, we therefore have $\P_{P_Y}\{r(Y)\geq r\}\geq 1-\alpha_0$.

    As the last step, we need to relate the target quantity $I$ to the newly defined $r(Y)$. Note that $\sum_{j=1}^m C_j = n$ by construction. We therefore have
    \[I = \sum_{j=1}^m \one_{C_j>0} = \sum_{j=1}^m \left(C_j - (C_j - 1)_+\right) = n - r(Y),\]
    which completes the proof.
\end{proof}

Theorem~\ref{thm:piecewise_constant_example_extension} offers a more informative lower bound than our original construction, in Theorem~\ref{thm:piecewise_constant_example}, both of which are valid without placing a bound on the loss. However, by combining the proof ideas from Theorem~\ref{thm:low_complexity_trunc_loss} (which uses a truncated loss) and Theorem~\ref{thm:piecewise_constant_example_extension}, we can obtain the following potentially even more informative lower bound.\bigskip

\begin{theorem}\label{thm:piecewise_constant_example_extension_trunc_loss}
    Fix $\alpha \in (0,1)$, $n \geq 1$, and $m \geq 1$. Let $\alpha_0 + \alpha_1 = \alpha$, with $\alpha_0, \alpha_1 > 0$. Fix any $B>0$, and define $\Delta_{n,B} \in (0,1]$ as the unique solution of 
    \[
    -\Delta_{n,B} - \log(1-\Delta_{n,B})=\frac{B\log(1/\alpha_1)}{n\hat{R}(\Fpwc^{(n-r)},\sample_n;B)},
    \]
    where $r$ is defined as in Theorem~\ref{thm:piecewise_constant_example_extension}, and where the truncated empirical risk $\hat{R}(\cdot,\sample_n;B)$ is defined as in Theorem~\ref{thm:low_complexity_trunc_loss}. Then
    \[
    \hat{L}^{\mathrm{pwc-trunc},r,B}_\alpha(\Fpwc^{(m)},\sample_n) := (1-\Delta_{n,B})\hat{R}(\Fpwc^{(n-r)},\sample_n;B)
    \]
    is a valid distribution-free lower bound on $R_P(\Fpwc^{(m)})$.
\end{theorem}

\begin{proof}[Proof of Theorem~\ref{thm:piecewise_constant_example_extension_trunc_loss}]
    By the same argument and using the same notation as in the proofs of Theorems~\ref{thm:piecewise_constant_example} and~\ref{thm:piecewise_constant_example_extension}, we have 
    $\P_P\{|I(\sample_n)|\leq n-r\}\geq 1-\alpha_0$.
    Let $\Delta'_{n,B}$ be the unique solution to
    \[
    -\Delta'_{n,B} - \log(1-\Delta'_{n,B})=\frac{B\log(1/\alpha_1)}{n\hat{R}(f_\varepsilon,\sample_n;B)}.
    \]
    Therefore, on the event $\{\hat{f}_\varepsilon \in \Fpwc^{(n-r)}\} \supseteq \{|I(\sample_n)|\leq n-r\}$, we have $\hat{R}(f_\varepsilon,\sample_n;B) = \hat{R}(\hat{f}_\varepsilon, \sample_n;B) \geq \hat{R}(\Fpwc^{(n-r)}, \sample_n;B)$ as in the proof of Theorem~\ref{thm:piecewise_constant_example_extension}, and therefore
    \[
      -\Delta_{n,B} - \log(1-\Delta_{n,B})=\frac{B\log(1/\alpha_1)}{n\hat{R}(\Fpwc^{(n-r)},\sample_n;B)} \geq \frac{B\log(1/\alpha_1)}{n\hat{R}(f_\varepsilon,\sample_n;B)} = -\Delta'_{n,B} - \log(1-\Delta'_{n,B}).
    \]
    Thus, on the event that $|I(\sample_n)|\leq n-r$, we have $\Delta'_{n,B} \leq \Delta_{n,B}$ and so
    \[
    \hat{L}^{\mathrm{pwc-trunc},r,B}_\alpha(\Fpwc^{(m)},\sample_n) \leq (1-\Delta_{n,B})\hat{R}(f_\varepsilon,\sample_n;B) \leq (1-\Delta'_{n,B})\hat{R}(f_\varepsilon,\sample_n;B).
    \]
    By Lemma~\ref{lemma:empirical_multiplicative_chernoff}, $\P_P\big\{(1-\Delta_{n,B}')\hat{R}(f_\varepsilon,\sample_n;B) \leq  \E[\hat{R}(f_\varepsilon,\sample_n;B)]\big\}\geq 1-\alpha_1$, and so
    \begin{multline*}
    \P_P\left\{\hat{L}_\alpha^{\mathrm{pwc-trunc},r,B}(\Fpwc^{(m)},\sample_n) \leq \E[\hat{R}(f_\varepsilon,\sample_n;B)]\right\} \\
    \geq  \P_P\big\{(1-\Delta_{n,B}')\hat{R}(f_\varepsilon,\sample_n;B) \leq  \E[\hat{R}(f_\varepsilon,\sample_n;B)]\textnormal{ and }|I(\sample_n)|\leq n-r\big\} \geq 1-\alpha.
    \end{multline*}
    As in the proof of Theorem~\ref{thm:low_complexity_trunc_loss}, we have $\E[\hat{R}(f_\varepsilon,\sample_n;B)]\leq R_P(\Fpwc^{(m)})+\varepsilon$, so it therefore holds that $\P_P\left\{\hat{L}_\alpha^{\mathrm{pwc-trunc},r,B}(\Fpwc^{(m)},\sample_n) \leq R_P(\Fpwc^{(m)}) + \varepsilon\right\}\geq 1-\alpha$, and since $\varepsilon>0$ can be taken to be arbitrarily small, this completes the proof.
\end{proof}

\section{Proofs and extensions for results in Section~\ref{sec:linear}}

\subsection{Proofs for results on the linear model example}
In this section, we prove all results stated in Section~\ref{sec:linear} for the linear model example.
\begin{proof}[Proof of Theorem~\ref{thm:linear_example}]
First, for any $Q\in\Qcal_{\textnormal{lin}}$, we have $R_Q(\Flin^{(d)})=0$ by definition. Therefore, by distribution-free validity of $\hat{L}_\alpha(\Flin^{(d)},\cdot)$, we must have
\[\P_{\sample_n\sim Q^n}\left\{\hat{L}_\alpha(\Flin^{(d)},\sample_n)=0\right\}\geq 1-\alpha.\]
Since this is true for every $Q^n\in \Qcal^{(n)}_{\textnormal{lin}}$, it must therefore also hold for any distribution in the convex hull of this class of distributions, i.e.,
\[\P_{\sample_n\sim Q}\left\{\hat{L}_\alpha(\Flin^{(d)},\sample_n)=0\right\}\geq 1-\alpha\textnormal{ for all $Q\in\tilde{\Qcal}^{(n)}_{\textnormal{lin}}$}.\]
By definition of total variation distance, therefore,
\[\P_{\sample_n\sim P^n}\left\{\hat{L}_\alpha(\Flin^{(d)},\sample_n) =0\right\}\geq 1-\alpha - \dtv(P^n,Q)\textnormal{ for all $Q\in\tilde{\Qcal}^{(n)}_{\textnormal{lin}}$},\]
which completes the proof.
\end{proof}

\begin{proof}[Proof of Proposition~\ref{prop:general-bound-lambda}]
Since $\tilde{\Qcal}_{\mathrm{lin}}^{(n)}$ is defined by taking mixtures of noiseless linear models, the quantity $\lambda_{n,d}(P)$ is invariant to taking linear transformations of the features $X$---that is, for any invertible $A\in\R^{d\times d}$, if $P'$ is the distribution of $(A X,Y)$ when we draw $(X,Y)\sim P$, then $\lambda_{n,d}(P') = \lambda_{n,d}(P)$. Therefore, without loss of generality, from this point on we can assume $\Omega = \bI_d$, and we will write $h(\bx) = h_{\bI_d}(\bx) = (\det(\bx\bx^\top))^{-1/2}$.

First, we define two continuous distributions on $(\bX,\beta)\in\R^{n\times d}\times\R^d$. 
Let $f(x,y)$ denote the density of the distribution $P$ on $\R^d\times\R$. Now, fix any constant $c>0$, and define $\tilde{P}$ as the distribution with density 
\[g_{\tilde{P}}(\bx,b) = \prod_{i=1}^n f(x_i,x_i^\top b) \cdot (2\pi c)^{-\frac{d-n}{2}} e^{-\| b\|^2_2/(2c)} \cdot \left[ \sqrt{\det(\bx\bx^\top)}\cdot e^{\|\mathbf{P}_{\bx} b\|^2_2/(2c)} \right],\]
where $\mathbf{P}_{\bx} = \bx^\top(\bx\bx^\top)^{-1}\bx\in\R^{d\times d}$ is the projection matrix to the row span of $\bx$. (Note that we are implicitly assuming $\bx\bx^\top\in\R^{n\times n}$ is invertible---since we are defining a density, it is sufficient that this holds for almost every $\bx\in\R^{n\times d}$.) 
Next, define $\tilde{Q}$ as the distribution with density
\[g_{\tilde{Q}}(\bx,b) = \prod_{i=1}^n f(x_i,x_i^\top b) \cdot (2\pi c)^{-\frac{d-n}{2}} e^{-\| b\|^2_2/(2c)} \cdot C^{-1},\]
where
\begin{equation}\label{eqn:C_def_and_bound}
    C = \E_{\tilde{P}}\left[ (\det(\bX\bX^\top))^{-1/2}\exp\left\{-\frac{1}{2c}\|\mathbf{P}_{\bX}\beta\|^2_2\right\}\right] \leq \E_{\tilde{P}}\left[h(\bX)\right],
\end{equation}
where the inequality holds by the definition of $h(\cdot)$.  (We will verify below that these functions are indeed well-defined densities, i.e., that $C$ is finite and positive, and that $g_{\tilde{P}}$ and $g_{\tilde{Q}}$ each integrate to $1$.)
Next, let $\tilde{P}_{(\bX,\bY)}$ denote the distribution of $(\bX,\bY):=(\bX,\bX\beta)\in\R^{n\times d}\times\R^n$ induced by drawing $(\bX,\beta)\sim\tilde{P}$, and define $\tilde{Q}_{(\bX,\bY)}$ analogously.

The intuition of the remainder of the proof is the following. First, in Step 1, we will verify
that $\tilde{P}$ and $\tilde{Q}$ are close in total variation distance---that is,
we will bound $\dtv(\tilde{P},\tilde{Q})$, 
which therefore induces a bound on $\dtv(\tilde{P}_{(\bX,\bY)},\tilde{Q}_{(\bX,\bY)})$, by construction
of $\tilde{P}_{(\bX,\bY)},\tilde{Q}_{(\bX,\bY)}$ from $\tilde{P},\tilde{Q}$.
Next, in Step 2, we will verify that
\begin{equation}\label{eqn:XY_marginal}
\textnormal{$\tilde{P}_{(\bX,\bY)}=P^n$, and $\tilde{Q}_{(\bX,\bY)}\in\tilde{\Qcal}_{\mathrm{lin}}^{(n)}$},
\end{equation}
and consequently, this will mean that $\lambda_{n,d}(P) \leq \dtv(\tilde{P}_{(\bX,\bY)},\tilde{Q}_{(\bX,\bY)})$. Finally, in Step 3, we will take a limit as $c\to\infty$ to complete the proof.

\paragraph{Step 1: bounding the total variation distance.}
We calculate
\begin{align*}
    \dtv(\tilde{P},\tilde{Q})
    &=\E_{\tilde{P}}\left[\left(1 - \frac{g_{\tilde{Q}}(\bX,\beta)}{g_{\tilde{P}}(\bX,\beta)}\right)_+\right]\\
    &=\E_{\tilde{P}}\left[\left(1 - \frac{C^{-1}}{\sqrt{\det(\bX\bX^\top)}\cdot e^{\|\mathbf{P}_{\bX}\beta\|^2_2/(2c)}}\right)_+\right]\\
    &\leq\E_{\tilde{P}}\left[\left(1 - \frac{h(\bX)}{\E_{\tilde{P}}[h(\bX)]} \cdot e^{-\|\mathbf{P}_{\bX}\beta\|^2_2/(2c)}\right)_+\right]\textnormal{\quad by definition of $h(\cdot)$ and by~\eqref{eqn:C_def_and_bound}}\\
    &\leq \E_{\tilde{P}}\left[e^{-\|\mathbf{P}_{\bX}\beta\|^2_2/(2c)} \cdot \left(1 - \frac{h(\bX)}{\E_{\tilde{P}}[h(\bX)]}\right)_+\right] + \E_{\tilde{P}}\left[1 - e^{-\|\mathbf{P}_{\bX}\beta\|^2_2/(2c)}\right]\\
    &\leq \E_{\tilde{P}}\left[\left(1 - \frac{h(\bX)}{\E_{\tilde{P}}[h(\bX)]}\right)_+\right] + \E_{\tilde{P}}\left[1 - e^{-\|\mathbf{P}_{\bX}\beta\|^2_2/(2c)}\right]\\
    &= \frac{1}{2}\E_{\tilde{P}}\left[\left|\frac{h(\bX)}{\E_{\tilde{P}}[h(\bX)]} - 1\right|\right] + \E_{\tilde{P}}\left[1 - e^{-\|\mathbf{P}_{\bX}\beta\|^2_2/(2c)}\right],
\end{align*}
where the last step holds since for any random variable $A$ with $\E[A]=1$, it holds by symmetry that $\E[(1-A)_+]=\E[(A-1)_+] = \frac{1}{2}\E[|A-1|]$.

Next, we have that $\tilde{P}_{(\bX,\bY)}$-almost surely
\[\mathbf{P}_{\bX}\beta = \bX^\top(\bX\bX^\top)^{-1}\bX\beta = \bX^\top(\bX\bX^\top)^{-1}\bY, \]
by definition of $\bY=\bX\beta$. From \eqref{eqn:XY_marginal} (which we will verify below), we know that $P^n=\tilde{P}_{(\bX,\bY)}$, and therefore,
\begin{equation}\label{eqn:step1_lambda_proof}
\dtv(\tilde{P},\tilde{Q})\leq\frac{1}{2}\E_P\left[\left|\frac{h(\bX)}{\E_P[h(\bX)]} - 1\right|\right] + \E_P\left[1 - e^{-\|\bX^\top(\bX\bX^\top)^{-1}\bY\|^2_2/(2c)}\right].
\end{equation}

\paragraph{Step 2: proving~\eqref{eqn:XY_marginal}.}
We now need to verify~\eqref{eqn:XY_marginal} (and, along the way, to check that $g_{\tilde{P}}$ and $g_{\tilde{Q}}$ are well-defined densities).
We begin by considering a joint distribution $\tilde{P}_*$ on $(\mathbf{X},\mathbf{Y},\beta)\in \R^{n\times d} \times \R^n\times \R^d$ generated as follows:
first sample $(\bX,\bY)\sim P^n$, then sample
    \[\beta \mid (\bX,\bY)\sim\mathcal{N}\left(\bX^\top(\bX\bX^\top)^{-1}\bY,c\mathbf{P}_{\bX}^\perp\right),\]  where $\mathbf{P}_{\bX}^\perp = \bI_d - \bX^\top (\bX\bX^\top)^{-1}\bX = \bI_d - \mathbf{P}_{\bX}$ denotes the projection matrix onto the orthogonal complement of the row span of $\bX$. 
    (Since the distribution $P$ has a density, note that $\bX$ has rank $n$, almost surely, and so $\bX\bX^\top$ is invertible, almost surely.)
Now consider the joint distribution of $(\bX,\beta)$ under $\tilde{P}_*$ (i.e., we are marginalizing out $\bY$), by first calculating the conditional density of $\beta\mid\bX$. Let $U\in\R^{d\times (d-n)}$ be an orthonormal matrix satisfying $UU^\top = \mathbf{P}_{\bX}^\perp$. Define
\[\gamma =  \left(\begin{array}{c} \bX \\ U^\top \end{array}\right)\cdot \beta.\]
Then by construction,
\begin{align*}
\gamma\mid (\bX,\bY) &\sim \mathcal{N}\left(\left(\begin{array}{c} \bX \\ U^\top \end{array}\right)\cdot \bX^\top(\bX\bX^\top)^{-1}\bY,\left(\begin{array}{c} \bX \\ U^\top \end{array}\right)\cdot c\mathbf{P}_{\bX}^\perp\cdot \left(\begin{array}{c} \bX \\ U^\top \end{array}\right)^\top\right)  \\
&=\mathcal{N}\Biggl(\left(\begin{array}{c} \bY \\ \mathbf{0}\end{array}\right), \left(\begin{array}{cc} \mathbf{0} & \mathbf{0}\\ \mathbf{0} & c\bI_{d-n}\end{array}\right)\Biggr).
\end{align*}
Then we can calculate that the distribution of $\gamma$ conditional on $\bX$ has the following density on $\R^d$ (note that we are now marginalizing over $\bY$):
\[g_{\gamma\mid \bX}(z\mid \mathbf{X}) = \prod_{i=1}^n f_{Y\mid X}(z_i\mid X_i) \cdot (2\pi c)^{-\frac{d-n}{2}}e^{-\|(z_{n+1},\dots,z_d)\|^2_2/(2c)}.\]
(Here $f_{Y|X}$ denotes the conditional density of $Y$ given $X$ under the joint distribution $(X,Y)\sim P$.) Moreover, when we condition on $\bX$, since $\beta$ is simply a linear transformation of $\gamma$, we therefore have
\[g_{\beta|\bX}(b\mid \bX) = \left|\det\Biggl(\biggl(\begin{array}{c} \bX \\ U^\top \end{array}\biggr)\Biggr)\right| \cdot g_{\gamma|\bX}\Biggl(\left(\begin{array}{c} \bX \\ U^\top \end{array}\right) \cdot b\Biggr).\]
We calculate
\begin{multline*}\left|\det\Biggl(\left(\begin{array}{c} \bX \\ U^\top \end{array}\right)\Biggr)\right| = \sqrt{\det\left(\left(\begin{array}{c} \bX \\ U^\top \end{array}\right)\left(\begin{array}{c} \bX \\ U^\top \end{array}\right)^\top\right)}\\
= \sqrt{\det\Biggl(\left(\begin{array}{cc}\bX\bX^\top & \bX U\\U^\top\bX^\top & U^\top U\end{array}\right)\Biggr)}
= \sqrt{\det\Biggl(\left(\begin{array}{cc}\bX\bX^\top & \mathbf{0}\\\mathbf{0} & \mathbf{I}_{d-n}\end{array}\right)\Biggr)}
= \sqrt{\det(\bX\bX^\top)}.\end{multline*}
And,
\begin{align*}g_{\gamma|\bX}\Biggl(\left(\begin{array}{c} \bX \\ U^\top \end{array}\right) \cdot b\Biggr) &= \prod_{i=1}^n f_{Y|X}(X_i^\top b\mid X_i) \cdot (2\pi c)^{-\frac{d-n}{2}} e^{-\|U^\top b\|^2_2/(2c)}\\
&=\prod_{i=1}^n f_{Y|X}(X_i^\top b\mid X_i) \cdot (2\pi c)^{-\frac{d-n}{2}} e^{-\|\mathbf{P}_{\bX}^\perp b\|^2_2/(2c)},\end{align*}
since $UU^\top = \mathbf{P}_{\bX}^\perp$.
Therefore, 
\begin{align*}g_{\beta|\bX}(b\mid \bX)
&= \sqrt{\det(\bX\bX^\top)}\cdot \prod_{i=1}^n f_{Y|X}(X_i^\top b\mid X_i) \cdot (2\pi c)^{-\frac{d-n}{2}} e^{-\|\mathbf{P}_{\bX}^\perp b\|^2_2/(2c)}\\
&=\prod_{i=1}^n f_{Y|X}(X_i^\top b\mid X_i) \cdot (2\pi c)^{-\frac{d-n}{2}} e^{-\|b\|^2_2/(2c)} \cdot\left[\sqrt{\det(\bX\bX^\top)}\cdot  e^{\|\mathbf{P}_{\bX} b\|^2_2/(2c)}\right].
\end{align*}
Writing $f_X$ as the marginal density of $X$ under the joint distribution $(X,Y)\sim P$, we then see that the density of $(\bX,\beta)$ under $\tilde{P}_*$ (i.e., after marginalizing out $\bY$) is given by
\begin{multline*}\prod_{i=1}^n f_X(x_i) \cdot g_{\beta|\bX}(b\mid \bx)
\\= \prod_{i=1}^n f_X(x_i) \cdot \prod_{i=1}^n f_{Y|X}(x_i^\top b\mid x_i) \cdot (2\pi c)^{-\frac{d-n}{2}} e^{-\|b\|^2_2/(2c)} \cdot\left[\sqrt{\det(\bx\bx^\top)}\cdot  e^{\|\mathbf{P}_{\bx} b\|^2_2/(2c)}\right]\\ = g_{\tilde{P}}(\bx,b).\end{multline*}
In particular, this verifies that $g_{\tilde{P}}$ is a well-defined density, and also verifies that $\tilde{P}$ is the marginal distribution of $(\bX,\beta)$, when $(\bX,\bY,\beta)\sim\tilde{P}_*$. Moreover, under the distribution $\tilde{P}_*$, we have $(\bX,\bY)\sim P^n$ by construction, and we also have $\bY = \bX\beta$ almost surely and therefore $(\bX,\bX\beta)\sim P^n$ also holds. This proves the first part of~\eqref{eqn:XY_marginal}.

Next consider $\tilde{Q}$. First, we verify that $C$ is finite, since 
\[C \leq \E_{\tilde{P}}[h(\bX)] = \E_P[h(\bX)]<\infty,\]
where the first step holds by~\eqref{eqn:C_def_and_bound}, the second holds by the first part of~\eqref{eqn:XY_marginal}, and the third holds by assumption in the proposition. Moreover,
$C>0$ since it is the expected value of a random variable that is positive almost surely.
Next, note that
\[g_{\tilde{Q}}(\bx,b) = g_{\tilde{P}}(\bx,b) \cdot C^{-1} \cdot \frac{1}{\sqrt{\det(\bx\bx^\top)}\cdot e^{\|\mathbf{P}_{\bx} b\|^2_2/(2c)}},\]
by definition of $g_{\tilde{P}},g_{\tilde{Q}}$. 
Plugging in the definition of $C$, we have
\[g_{\tilde{Q}}(\bx,b) = g_{\tilde{P}}(\bx,b) \cdot \frac{ (\det(\bx\bx^\top))^{-1/2}\exp\left\{-\frac{1}{2c}\|\mathbf{P}_{\bx}b\|^2_2\right\}}{\E_{\tilde{P}}\left[ (\det(\bX\bX^\top))^{-1/2}\exp\left\{-\frac{1}{2c}\|\mathbf{P}_{\bX}b\|^2_2\right\}\right]},\]
which means that $g_{\tilde{Q}}$ must also integrate to $1$ (i.e., it is a well-defined density).
Next we need to verify that $\tilde{Q}_{(\bX,\bY)}\in\tilde{\Qcal}_{\mathrm{lin}}^{(n)}$. Define $\tilde{Q}_\beta$ as the distribution of $(\bX,\bY)$ conditional on $\beta$, under the joint distribution $\tilde{Q}$ on $(\bX,\beta)$, when we define $\bY=\bX\beta$. Clearly, $\tilde{Q}_{(\bX,\bY)}$ can be expressed as a mixture of such distributions. Moreover, the data points $(X_i,Y_i)$ are i.i.d.\ conditional on $\beta$ (since by examining the density $g_{\tilde{Q}}$ we can see that it factors over data points $i$), and satisfy $Y_i = X_i^\top\beta$ almost surely. Therefore, $\tilde{Q}_\beta\in\Qcal_{\mathrm{lin}}^{(n)}$, which completes the proof of~\eqref{eqn:XY_marginal}.

\paragraph{Step 3: combining everything.}
Combining our calculations so far, we have shown that
\begin{multline*}\lambda_{n,d}(P) \leq \dtv(\tilde{P}_{(\bX,\bY)},\tilde{Q}_{(\bX,\bY)}) \leq \dtv(\tilde{P},\tilde{Q})\\ \leq \frac{1}{2}\E_P\left[\left|\frac{h(\bX)}{\E_P[h(\bX)]} - 1\right|\right] + \E_P\left[1 - e^{-\|\bX^\top(\bX\bX^\top)^{-1}\bY\|^2_2/(2c)}\right],\end{multline*}
where the first step holds by~\eqref{eqn:XY_marginal}, the second step holds by construction of $\tilde{P}_{(\bX,\bY)},\tilde{Q}_{(\bX,\bY)}$ from $\tilde{P},\tilde{Q}$, and the third step holds by~\eqref{eqn:step1_lambda_proof}.
 Moreover, since this is true for any $c>0$, and since for any random variable $A\geq 0$ it holds that $\lim_{c\to\infty}\E[e^{-A/c}]=1$ by the dominated convergence theorem,
we therefore have
\[\lambda_{n,d}(P)\leq  \frac{1}{2}\E_P\left[\left|\frac{h(\bX)}{\E_P[h(\bX)]} - 1\right|\right] ,\]
as desired.
\end{proof}

\begin{proof}[Proof of Corollary~\ref{cor:Gaussian-error}]
By Proposition~\ref{prop:general-bound-lambda} (applied with $\Omega=\Sigma^{-1}$), we have 
\begin{align*}
    \lambda_{n,d}(P)
    &\leq \frac{1}{2}\E_P\left[\left|\frac{h_{\Omega}(\bX)}{\E_P[h_{\Omega}(\bX)]} - 1\right|\right]\\
    &\leq \frac{1}{\sqrt{2}}\left\{ \E_P\left[\log\left(\frac{\E_P[h_{\Omega}(\bX)]}{h_{\Omega}(\bX)}\right)\right]\right\}^{1/2}\\
    &=\frac{1}{\sqrt{2}}\left\{ \log\left(\E_P\left[(\det(\bX\Sigma^{-1}\bX^\top))^{-1/2}\right]\right) - \E_P\left[\log\left( (\det(\bX\Sigma^{-1}\bX^\top))^{-1/2}\right)\right]\right\}^{1/2},
\end{align*}
where the second step holds by Pinsker's inequality (to be more concrete, we are using the fact that, for a random variable $Z\geq 0$ with $\E[Z]=1$, $\frac{1}{2}\E[|Z-1|] \leq \sqrt{\frac{1}{2}\E[\log(1/Z)]}$).
Next we compute each of these expected values. First, observe that $\bX\Sigma^{-1}\bX^\top\sim W_n(\bI_n,d)$ (a Wishart distribution).
\cite{goodman1963distribution} proves that the determinant of a random matrix drawn from the Wishart distribution $W_n(\bI_n, d)$ is distributed as the product of independent random variables with a $\chi^2$-distribution, i.e.,
\[\det(\bX\Sigma^{-1}\bX^\top) \stackrel{\mathrm{d}}{=} G_d \cdot G_{d-1}\cdot \hdots \cdot G_{d-n+1},\]
where $G_k\sim \chi^2_k$ for each $k$, and the $G_k$'s are mutually independent.
It is straightforward to check that, since $G_k\sim \chi^2_k$, 
	\[
			\E[G_k^{-1/2}] = \int_0^\infty \frac{1}{\sqrt{x}} \frac{1}{2^{k/2} \Gamma(k/2) } x^{k/2-1} e^{-x/2}\;\mathsf{d} x =\frac{1}{\sqrt{2}} \frac{\Gamma( \frac{k-1}{2} )}{\Gamma( \frac{k}{2} )},
		\]
    and
        \begin{multline*}
        \E\bigl[\log G_k\bigr] - \log 2= \E\bigl[\log (G_k/2)\bigr] = \int_0^\infty \frac{\log(x/2)}{2^{k/2}\Gamma(k/2)}x^{k/2-1}e^{-x/2}\;\mathsf{d}x \\
        = \int_0^\infty \frac{\log(t)}{\Gamma(k/2)}t^{k/2-1}e^{-t}\;\mathsf{d}t = \psi\Bigl(\frac{k}{2}\Bigr), 
    \end{multline*}
where the final equality follows from the integral representation of the digamma function $\psi(u) := \frac{\mathsf{d}}{\mathsf{d}u}\log \Gamma(u) = \int_0^\infty \log(x)x^{u-1}e^{-x}\; \mathsf{d}x/\Gamma(u)$, $u > 0$ \citep{gordon1994gamma}.
Therefore,
\begin{align*}
    &\log \E_P[(\det(\bX\Sigma^{-1}\bX^\top))^{-1/2}] -\E_P\left[\log(\det(\bX\Sigma^{-1}\bX^\top)^{-1/2})\right]\\
    &=\log \E[(G_d\cdot\hdots\cdot G_{d-n+1})^{-1/2}] +\frac{1}{2} \E\left[\log(G_d\cdot\hdots\cdot G_{d-n+1})\right]\\
    &=\log \prod_{k=d-n+1}^d\E[G_k^{-1/2}] + \frac{1}{2}\sum_{k=d-n+1}^d\E[\log G_k]\\
    &= \sum_{k=d-n+1}^d \left\{  \log\left(\frac{1}{\sqrt{2}} \frac{\Gamma( \frac{k-1}{2} )}{\Gamma( \frac{k}{2} )}\right)  + \frac{1}{2}\psi\Bigl(\frac{k}{2}\Bigr) + \frac{\log 2}{2}\right\} \\
    &= \sum_{j=1}^n \left\{  \log\left(\frac{\Gamma( \frac{d-1}{2} + \frac{1-j}{2} )}{\Gamma( \frac{d}{2} + \frac{1-j}{2} )}\right)  + \frac{1}{2}\psi\Bigl(\frac{d}{2} + \frac{1-j}{2}\Bigr)\right\}.
\end{align*}
Next, we bound each term in the sum. Since $\psi(u) = \frac{\mathsf{d}}{\mathsf{d}u}\log \Gamma(u)$ and $d\geq n+2$, we have
\begin{align*}
&\log\left(\Gamma\left( \frac{d-1}{2} + \frac{1-j}{2} \right)\right) - \log\left(\Gamma\left( \frac{d}{2} + \frac{1-j}{2} \right)\right)+ \frac{1}{2}\psi\Bigl(\frac{d}{2} + \frac{1-j}{2}\Bigr)\\
&=\frac{1}{2}\left[\psi\Bigl(\frac{d}{2} + \frac{1-j}{2}\Bigr)-\psi\Bigl(\frac{d-c}{2} + \frac{1-j}{2}\Bigr)\right]\textnormal{\quad for some $c\in[0,1]$, by Taylor's theorem}\\
&\leq \frac{1}{2}\left[\psi\Bigl(\frac{d}{2} + \frac{1-j}{2}\Bigr)-\psi\Bigl(\frac{d-1}{2} + \frac{1-j}{2}\Bigr)\right]\textnormal{\quad since $\psi(u)$ is an increasing function on $u>0$}\\
&\leq \frac{1}{4}\left[\psi\Bigl(\frac{d}{2} + \frac{1-j}{2}\Bigr)-\psi\Bigl(\frac{d-2}{2} + \frac{1-j}{2}\Bigr)\right]\textnormal{\quad since $\psi(u)$ is a concave function on $u>0$}\\
&= \frac{1}{4}\left[ \frac{1}{\frac{d-2}{2} + \frac{1-j}{2}}\right] = \frac{1}{2(d-1-j)},
\end{align*}
where we use the fact that $\psi(u+1)=\psi(u)+\frac{1}{u}$ holds for all $u>0$. Therefore, combining all our calculations,
\begin{multline*}
2\big(\lambda_{n,d}(P)\big)^2 \leq 
    \log \E_P[(\det(\bX\Sigma^{-1}\bX^\top))^{-1/2}] +\frac{1}{2} \E_P\left[\log(\det(\bX\Sigma^{-1}\bX^\top))\right]\\
    \leq\sum_{j=1}^n \frac{1}{2(d-1-j)} \leq \frac{n}{2(d-1-n)},
\end{multline*}
 which completes the proof.

\end{proof}

\subsection{Extensions for the linear model example: generalizing to other distributions} \label{app:linear-model-general-class}
We extend the results of Section~\ref{sec:linear}, to prove a bound on $\lambda_{n,d}(P)$ for a broader family of distributions.
\begin{proposition} \label{prop:lambda_extension}
	Let $P,Q$ be distributions on $\R^d\times \R$, such that $Q$ is absolutely continuous with respect to $P$. Assume $\frac{\mathsf{d}Q}{\mathsf{d}P}(x,y) \leq \varepsilon^{-1}$ for $P$-almost every $(x,y)$. Then, for all $n\geq 1$
	\begin{equation*}
		\lambda_{n,d}(Q) \leq \lambda_{\lceil2n/ \varepsilon\rceil,d }(P) + e^{-n/4}.
	\end{equation*}
\end{proposition} 
For instance, if $P$ is a distribution with a Gaussian marginal $P_X$, then we know that $\lambda_{\lceil2n/ \varepsilon\rceil,d }(P)\approx 0$ (as long as $d\gg n/ \varepsilon$), by Corollary~\ref{cor:Gaussian-error}. This means that $\lambda_{n,d}(Q)$ will also be small for any distribution $Q$ with sufficiently light tails, and therefore,  Theorem~\ref{thm:linear_example} shows that a nontrivial lower bound is not possible for the model class risk $R_Q(\Flin^{(d)})$.

In order to prove this result, first we need a lemma on rejection sampling.
\begin{lemma}\label{lem:rejection_sampling}
    Let $P,Q$ be distributions on $\Zcal$, such that $Q$ is absolutely continuous with respect to $P$. Assume $\frac{\mathsf{d}Q}{\mathsf{d}P}(z) \leq \varepsilon^{-1}$ for $P$-almost every $z$. 
    
    Let $\tilde{P}$ denote the distribution on $(Z,B)\in\Zcal\times\{0,1\}$ constructed by sampling $Z\sim P$, then sampling $B\mid Z\sim\textnormal{Bernoulli}(\varepsilon\cdot\frac{\mathsf{d}Q}{\mathsf{d}P}(Z) )$.  
    Then, for any $N\geq n\geq 1$, and for any function $f:\Zcal^n\to[0,1]$,
    \[0\leq \E_{Q^n}[f] - \E_{\tilde{P}^N}\left[ \sum_{1\leq i_1<\dots<i_n\leq N} f(Z_{i_1},\dots,Z_{i_n}) \cdot \frac{\one_{B_{i_1} = \dots = B_{i_n}=1}}{1\vee {\sum_i B_i\choose n}} \right] \\{}\leq \P\{\textnormal{Binom}(N,\varepsilon)<n\}.\]
\end{lemma}
To understand this lemma, we can interpret it as a result about rejection sampling. To draw one sample from the distribution of $f(Z_1,\dots,Z_n)$ under $Q^n$, we do the following:
\begin{itemize}
    \item We draw $N$ samples from $P$, given by $Z_1,\dots,Z_N$, and draw the Bernoulli random variables $B_1,\dots,B_N$ to perform rejection sampling.
    \item The ``accepted'' draws---that is, all $Z_i$ for which $B_i= 1$---can be viewed as a random sample drawn i.i.d.\ from $Q$. If $\sum_{i=1}^N B_i \geq n$, then, we can choose a random subset of $n$ of these accepted draws (indices $1\leq i_1<\dots<i_n\leq N$), and evaluate $f(Z_{i_1},\dots,Z_{i_n})$. This is a draw from the target distribution.
    \item However, on the event that $\sum_{i=1}^N B_i < n$ (which is highly unlikely if we choose $N$ to be sufficiently large), we do not have sufficiently many accepted samples; instead we might simply return $0$ since no estimate is available.
\end{itemize}
With this intuition in place, the proof is straightforward.
\begin{proof}[Proof of Lemma~\ref{lem:rejection_sampling}]
    First we condition on $B_1,\dots,B_N$. On the event that $\sum_{i=1}^N B_i \geq n$, by the argument described above, for any indices $1\leq i_1<\dots<i_n\leq N$ such that $B_{i_1}=\dots=B_{i_n}=1$, it holds that $f_{i_1,\dots,i_n} := f(Z_{i_1}, \ldots, Z_{i_n})$ is therefore a draw from the target distribution---that is,
    \[\E_{\tilde{P}^N}\left[f_{i_1,\dots,i_n} \mid B_{i_1} = \dots = B_{i_n}=1\right] = \E_{Q^n}[f].\]
    Then
    \[\E_{\tilde{P}^N}\left[f_{i_1,\dots,i_n} \cdot\one_{B_{i_1} = \dots = B_{i_n}=1} \mid B_1,\dots,B_N\right] = \E_{Q^n}[f] \cdot\one_{B_{i_1} = \dots = B_{i_n}=1},\]
    and so averaging over all possible collections of indices, we then have
    \[\frac{\E_{\tilde{P}^N}\left[\sum_{1\leq i_1<\dots <i_n\leq N} f_{i_1,\dots,i_n} \cdot\one_{B_{i_1} = \dots = B_{i_n}=1}\,\middle|\, B_1,\dots,B_N\right]}{\sum_{1\leq i_1<\dots <i_n\leq N}\one_{B_{i_1} = \dots = B_{i_n}=1}} = \E_{Q^n}[f],\]
    on the event that the sum in the denominator is positive. Note that this denominator is equal to ${\sum_i B_i\choose n}$, which is nonzero if and only if $\sum_i B_i \geq n$, so equivalently we have
    \[\frac{\E_{\tilde{P}^N}\left[\sum_{1\leq i_1<\dots <i_n\leq N} f_{i_1,\dots,i_n} \cdot\one_{B_{i_1} = \dots = B_{i_n}=1}\,\middle|\, B_1,\dots,B_N\right]}{{\sum_i B_i\choose n}} = \E_{Q^n}[f],\]
    on the event that $\sum_i B_i \geq n$. To cover both cases, then, we have
    \[\frac{\E_{\tilde{P}^N}\left[\sum_{1\leq i_1<\dots <i_n\leq N} f_{i_1,\dots,i_n} \cdot\one_{B_{i_1} = \dots = B_{i_n}=1}\,\middle|\, B_1,\dots,B_N\right]}{1 \vee {\sum_i B_i\choose n}} = \E_{Q^n}[f] \cdot\one\left\{\sum_i B_i \geq n\right\}.\]
    Marginalizing over the $B_i$'s, then,
    \[\E_{\tilde{P}^N}\left[\frac{\sum_{1\leq i_1<\dots <i_n\leq N} f_{i_1,\dots,i_n} \cdot\one_{B_{i_1} = \dots = B_{i_n}=1}}{1 \vee {\sum_i B_i\choose n}}\right] = \E_{Q^n}[f] \cdot\P_{\tilde{P}^N}\left\{\sum_i B_i \geq n\right\}.\]
    Therefore,
    \begin{multline*}\E_{Q^n}[f] - \E_{\tilde{P}^N}\left[\frac{\sum_{1\leq i_1<\dots <i_n\leq N} f_{i_1,\dots,i_n} \cdot\one_{B_{i_1} = \dots = B_{i_n}=1}}{1 \vee {\sum_i B_i\choose n}}\right]\\ = \E_{Q^n}[f] \cdot\P_{\tilde{P}^N}\left\{\sum_i B_i < n\right\}  = \E_{Q^n}[f]\cdot \P\{\textnormal{Binom}(N,\varepsilon)<n\}\\\in \Big[0,\P\{\textnormal{Binom}(N,\varepsilon)<n\}\Big].\end{multline*}

\end{proof}

With this lemma in place, we are now ready to prove the extension.
\begin{proof}[Proof of Proposition~\ref{prop:lambda_extension}]
First we will prove the result in the case where $\frac{\mathsf{d}Q}{\mathsf{d}P}(x,y)>0$ for all $(x,y)\in\R^d\times\R$.

Fix any $N\geq n$ and any $\delta > 0$. By the definition of $\lambda_{N,d}(P)$, there exists a distribution $P_*\in \tilde{\Qcal}_{\textnormal{lin}}^{(N)}$, such that 
    \[
        \dtv( P^N, P_* ) \leq \lambda_{N,d}( P ) + \delta.
    \]
    By definition of $\tilde{\Qcal}_{\textnormal{lin}}^{(N)}$, the distribution $P_*$ is equal to a mixture of distributions in $\Qcal_{\textnormal{lin}}^{(N)}$, i.e., there exists a probability measure $\nu$ on the space of distributions $\Qcal_{\textnormal{lin}}$, such that sampling from $P_*$ is the same as sampling from $(P_0)^N$ after sampling $P_0 \sim \nu$.

For any $P_0\sim \nu$, define a corresponding distribution $Q(P_0)$ given by \[\frac{\mathsf{d}Q(P_0)}{\mathsf{d}P_0}(x,y) = \frac{\frac{\mathsf{d}Q}{\mathsf{d}P}(x,y)}{\E_{(X,Y)\sim P_0}\left[\frac{\mathsf{d}Q}{\mathsf{d}P}(X,Y)\right]}.\]
Since $P_0\in \Qcal_{\mathrm{lin}}$ and $Q(P_0)$ is absolutely continuous with respect to $P_0$, it follows that $Q(P_0)\in\Qcal_{\mathrm{lin}}$ as well (i.e., since there exists some $\beta\in\R^d$ such that $Y=X^\top\beta$ holds $P_0$-almost surely, we also have that $Y=X^\top\beta$ holds $Q(P_0)$-almost surely). Consequently, we can define $Q_*\in\tilde{\Qcal}_{\mathrm{lin}}^{(n)}$ as the mixture obtained by sampling $P_0\sim \nu$ and then returning $(Q(P_0))^n$.

Our next step is to bound $\dtv(Q^n,Q_*)$. Fix any $A\subseteq(\R^d\times\R)^n$. 
By Lemma~\ref{lem:rejection_sampling} (applied with $f = \one_A$ and $Z=(X,Y)$),
\begin{multline*}Q^n(A) \leq \E_{\tilde{P}^N}\left[\frac{\sum_{1\leq i_1<\dots <i_n\leq N} \one_{((X_{i_1},Y_{i_1}),\dots,(X_{i_n},Y_{i_n}))\in A} \cdot\one_{B_{i_1} = \dots = B_{i_n}=1}}{1 \vee {\sum_i B_i\choose n}}\right]\\ {}+ \P\{\textnormal{Binom}(N,\varepsilon)<n\},\end{multline*}
where $\tilde{P}$ is the distribution on $(X,Y,B)\in\R^d\times\R\times\{0,1\}$ defined as in the statement of the lemma.
Next, for any $P_0\sim\nu$, define $\tilde{P}_0$ as the distribution on $(X,Y,B)\in\R^d\times\R\times\{0,1\}$ obtained by sampling $(X,Y)\sim P_0$, then \[B\mid (X,Y)\sim \textnormal{Bernoulli}\left(\varepsilon\cdot \frac{\mathsf{d}Q}{\mathsf{d}P}(X,Y)\right) = \textnormal{Bernoulli}\left(\varepsilon_{P_0} \cdot \frac{\mathsf{d}Q(P_0)}{\mathsf{d}P_0}(X,Y)\right) ,\] 
where $\varepsilon_{P_0} = \varepsilon\cdot \E_{(X,Y)\sim P_0}\left[\frac{\mathsf{d}Q}{\mathsf{d}P}(X,Y)\right]$. Again applying Lemma~\ref{lem:rejection_sampling},
\[(Q(P_0))^n(A) \geq \E_{(\tilde{P}_0)^N}\left[ \frac{\sum_{1\leq i_1<\dots <i_n\leq N} \one_{((X_{i_1},Y_{i_1}),\dots,(X_{i_n},Y_{i_n}))\in A} \cdot\one_{B_{i_1} = \dots = B_{i_n}=1}}{1 \vee {\sum_i B_i\choose n}} \right].\]
Therefore,
\begin{multline*}Q_*(A) = \E_{P_0\sim \nu}[(Q(P_0))^n(A)]
\\\geq\E_{P_0\sim\nu}\left[\E_{(\tilde{P}_0)^N}\left[ \frac{\sum_{1\leq i_1<\dots <i_n\leq N} \one_{((X_{i_1},Y_{i_1}),\dots,(X_{i_n},Y_{i_n}))\in A} \cdot\one_{B_{i_1} = \dots = B_{i_n}=1}}{1 \vee {\sum_i B_i\choose n}} \right]\right]. \end{multline*}
Let $\tilde{P}_*$ be the distribution on $(\R^d\times\R\times\{0,1\})^N$ obtained as follows: sample $P_0 \sim \nu$, then return a draw from $(\tilde{P}_0)^N$. Then we equivalently have
\[Q_*(A) \geq \E_{\tilde{P}_*}\left[\frac{\sum_{1\leq i_1<\dots <i_n\leq N} \one_{((X_{i_1},Y_{i_1}),\dots,(X_{i_n},Y_{i_n}))\in A} \cdot\one_{B_{i_1} = \dots = B_{i_n}=1}}{1 \vee {\sum_i B_i\choose n}} \right].\]
Since the quantity inside the expected value must always lie in $[0,1]$, it therefore holds that
\[Q_*(A) \geq \E_{\tilde{P}^N}\left[\frac{\sum_{1\leq i_1<\dots <i_n\leq N} \one_{((X_{i_1},Y_{i_1}),\dots,(X_{i_n},Y_{i_n}))\in A} \cdot\one_{B_{i_1} = \dots = B_{i_n}=1}}{1 \vee {\sum_i B_i\choose n}} \right] -\dtv(\tilde{P}^N,\tilde{P}_*).\]

Combining everything, then,
\[Q^n(A) \leq Q_*(A) + \dtv(\tilde{P}^N,\tilde{P}_*) + \P\{\textnormal{Binom}(N,\varepsilon)<n\}.\]
Since this holds for all $A\subseteq(\R^d\times\R)^n$, and since $Q_*\in\tilde{\Qcal}_{\textnormal{lin}}^{(n)}$, we therefore have
\[\lambda_{n,d}(Q) \leq \dtv(Q^n,Q_*) = \sup_{A\subseteq(\R^d\times\R)^n}\{Q^n(A)-Q_*(A)\}\leq \dtv(\tilde{P}^N,\tilde{P}_*) + \P\{\textnormal{Binom}(N,\varepsilon)<n\}.\]
Next, by construction, for both $\tilde{P}^N$ and $\tilde{P}_*$, the conditional distribution of 
$(B_1,\dots,B_N)$ given $(X_1,Y_1),\dots,(X_N,Y_N)$ is equal to
\[\textnormal{Bernoulli}\left(\varepsilon\cdot \frac{\mathsf{d}Q}{\mathsf{d}P}(X_1,Y_1)\right)\times \dots \times \textnormal{Bernoulli}\left(\varepsilon\cdot \frac{\mathsf{d}Q}{\mathsf{d}P}(X_N,Y_N)\right).\]
In other words, the total variation distance between $\tilde{P}^N$ and $\tilde{P}_*$ is solely due to the difference in marginal distributions over $((X_1,Y_1),\dots,(X_N,Y_N))$, i.e.,
\[\dtv(\tilde{P}^N,\tilde{P}_*) = \dtv(P^N,P_*).\]
By our choice of $P_*$, we therefore have $\dtv(\tilde{P}^N,\tilde{P}_*)= \dtv(P^N,P_*)\leq \lambda_{N,d}(P) + \delta$, and so
\[\lambda_{n,d}(Q) \leq \lambda_{N,d}(P) + \delta + \P\{\textnormal{Binom}(N,\varepsilon)<n\}.\]
Finally, 
taking $N = \lceil 2n /\varepsilon\rceil$, and applying a Chernoff bound \citep[Theorem~2.3(c)]{mcdiarmid1998concentration}, we have
\begin{equation*}
	\P\{\textnormal{Binom}(N,\varepsilon) < n\} \leq \exp\{-\varepsilon N/8\} \leq e^{-n/4},
\end{equation*}
and thus
\[\lambda_{n,d}(Q) \leq \lambda_{N,d}(P) + \delta + e^{-n/4}.\]
Since $\delta>0$ can be taken to be arbitrarily small, this completes the proof for the case that $\frac{\mathsf{d}Q}{\mathsf{d}P}(x,y)>0$ for all $(x,y)$.

Next we turn to the general case. Fix some small $c\in(0,1)$, and define $Q_c = (1-c) Q+ c P$. Then
\[\frac{\mathsf{d}Q_c}{\mathsf{d}P}(x,y) = (1-c)\cdot \frac{\mathsf{d}Q}{\mathsf{d}P}(x,y) + c \in (0,\varepsilon^{-1}],\]
so we can apply the result of the proposition to this perturbed distribution to obtain
\[\lambda_{n,d}(Q_c) \leq \lambda_{N,d}(P) + e^{-n/4}.\]
But by definition of $\lambda_{n,d}$, we have
\[\lambda_{n,d}(Q) \leq \lambda_{n,d}(Q_c) + \dtv(Q^n, (Q_c)^n) \leq \lambda_{n,d}(Q_c) +  nc.\]
Taking $c\in(0,1)$ to be arbitrarily small, we have completed the proof for the general case.

\end{proof}

\subsection{Extensions for the linear model example: a tighter bound for the low-dimensional case}
In this section, we construct a more informative lower bound for $R_P(\Flin^{(d)})$ by applying the truncated loss construction of Theorem~\ref{thm:low_complexity_trunc_loss}, and show that in the low-dimensional setting ($d<n$), this lower bound provides an accurate estimate of the model class risk $R_P(\Flin^{(d)})$.

\begin{theorem}\label{thm:linear_example_extension_trunc_loss}
    Fix $\alpha \in (0,1)$, and $n \geq d \geq 1$. Let $\alpha_0 + \alpha_1 = \alpha$, with $\alpha_0, \alpha_1 > 0$. 
    Define the valid lower bound $\hat{L}_\alpha^{\mathrm{ERM-trunc},B}(\Flin^{(d)},\cdot)$ as in Theorem~\ref{thm:low_complexity_trunc_loss}.
    Then, for any data set $\sample_n = ((X_1,Y_1),\dots,(X_n,Y_n))\in(\R^d\times\R)^n$, if it holds that
    \begin{equation}\label{eqn:concentration_for_lin_extension}\frac{1}{n}\sum_{i=1}^n Y_i^4\leq \gamma,  \textnormal{  \ and for all $b\in\R^d$, \ } \frac{1}{n}\sum_{i=1}^n (X_i^\top b)^2\geq \lambda_0\|b\|^2_2, \ \frac{1}{n}\sum_{i=1}^n (X_i^\top b)^4\leq \lambda_1\|b\|^4_2,\end{equation}
    for some $\gamma ,\lambda_0, \lambda_1 > 0$, then
    \[\hat{L}_\alpha^{\mathrm{ERM-trunc},B}(\Flin^{(d)},\sample_n)
    \geq \frac{1}{n}\|\mathcal{P}_{\bX}^\perp(\bY)\|^2_2 - \sqrt{\frac{2B^2\log(1/\alpha)}{n}} - \frac{c}{B},\]
    where $\mathcal{P}_{\bX}^\perp$ denotes projection to the orthogonal complement of the span of the columns of $\bX$, and where the constant $c$ depends only on $\gamma,\lambda_0,\lambda_1$.
\end{theorem}
The condition~\eqref{eqn:concentration_for_lin_extension} will hold with high probability under standard mild assumptions on the distribution of the data, via the usual concentration arguments (including matrix concentration arguments for the minimum eigenvalue---see, e.g., \citet{tropp2012user}).

As an illustrative example, we can consider the setting where $P$ follows a linear model, $Y_i = X_i^\top\beta^* + \zeta_i$, for some fixed true coefficient vector $\beta^*$, and i.i.d.\ noise $\zeta_i$ with mean zero and variance $\sigma^2$. In this case, we have $R_P(\Flin^{(d)}) = \sigma^2$. And, by standard arguments, when $\bX$ has full column rank, then \[\E_P\left[\frac{1}{n}\|\mathcal{P}_{\bX}^\perp(\bY)\|^2_2\right] = \left(1-\frac{d}{n}\right)\sigma^2 = \left(1-\frac{d}{n}\right)R_P(\Flin^{(d)}).\]
In particular, if $d\ll n$, and if we choose the truncation parameter $B$ to satisfy $1\ll B \ll \sqrt{n}$, this shows that the valid lower bound $\hat{L}_\alpha^{\mathrm{ERM-trunc},B}(\Flin^{(d)},\sample_n)$ is in fact an accurate estimator of the true risk $R_P(\Flin^{(d)})$. 
\begin{proof}[Proof of Theorem~\ref{thm:linear_example_extension_trunc_loss}]
    Applying~\eqref{ineq:easier-lower-bound-trunc}, we have
    \[\hat{L}_\alpha^{\mathrm{ERM-trunc},B}(\Flin^{(d)},\sample_n)  \geq \hat{R}(\Flin^{(d)},\sample_n;B)- \sqrt{\frac{2B^2\log(1/\alpha)}{n}}.\]
    Consequently, from this point on, we only need to show that
    $\hat{R}(\Flin^{(d)},\sample_n;B) \geq\frac{1}{n}\|\mathcal{P}_{\bX}^\perp(\bY)\|_2^2 - \frac{c}{B}$
    holds under the assumption~\eqref{eqn:concentration_for_lin_extension}.

    Consider any function of the form $f(x)=x^\top\beta$, for $\beta\in\R^d$. First we consider the case $\|\beta\|^2_2>\frac{4\sqrt{\gamma}}{\lambda_0}$. Define the unit vector $u = \beta/\|\beta\|_2$. Then
    \begin{align*}
        &\hat{R}(f,\sample_n;B)
        =\frac{1}{n}\sum_{i=1}^n \min\{(Y_i - X_i^\top \beta)^2, B\}\\
        &\geq \frac{1}{n}\sum_{i=1}^n \min\left\{\frac{1}{2}(X_i^\top\beta)^2 - Y_i^2, B\right\}\textnormal{\quad since $(a+b)^2\leq 2a^2+2b^2$ for all $a,b\in\R$}\\
        &\geq \frac{1}{n}\sum_{i=1}^n \left(\min\left\{\frac{1}{2}(X_i^\top\beta)^2,B\right\} - Y_i^2\right)\\
        &\geq  - \frac{1}{n}\sum_{i=1}^n Y_i^2+ \frac{1}{n}\sum_{i=1}^n \min\left\{\frac{2\sqrt{\gamma}}{\lambda_0}(X_i^\top u)^2,B\right\} \textnormal{\quad since $\|\beta\|^2_2 > \frac{4\sqrt{\gamma}}{\lambda_0}$}\\
      &\geq - \frac{1}{n}\sum_{i=1}^n Y_i^2 + \frac{1}{n}\sum_{i=1}^n \frac{2\sqrt{\gamma}}{\lambda_0}(X_i^\top u)^2 - \frac{1}{n}\sum_{i=1}^n  \frac{2\sqrt{\gamma}}{\lambda_0}(X_i^\top u)^2 \cdot \one\left\{\frac{2\sqrt{\gamma}}{\lambda_0}(X_i^\top u)^2 \geq B\right\} \\
        &\geq  - \frac{1}{n}\sum_{i=1}^n Y_i^2 + \frac{1}{n}\sum_{i=1}^n \frac{2\sqrt{\gamma}}{\lambda_0}(X_i^\top u)^2  - \frac{1}{n}\sum_{i=1}^n \frac{1}{B} \left[\frac{2\sqrt{\gamma}}{\lambda_0}(X_i^\top u)^2\right]^2\\
        &\geq  - \frac{1}{n}\sum_{i=1}^n Y_i^2 + 2\sqrt{\gamma} - \frac{4\gamma\lambda_1}{B\lambda_0^2}\textnormal{\quad by~\eqref{eqn:concentration_for_lin_extension}}\\
        &\geq \frac{1}{n}\sum_{i=1}^n Y_i^2- \frac{4\gamma\lambda_1}{B\lambda_0^2} \textnormal{\quad since $\frac{1}{n}\sum_{i=1}^nY_i^2\leq \sqrt{\frac{1}{n}\sum_{i=1}^nY_i^4}\leq \sqrt{\gamma}$ by~\eqref{eqn:concentration_for_lin_extension}}\\
        &\geq \frac{1}{n}\|\mathcal{P}_{\bX}^\perp(\bY)\|^2_2- \frac{4\gamma\lambda_1}{B\lambda_0^2}\textnormal{\quad since $\sum_{i=1}^n Y_i^2 = \|\bY\|^2_2 \geq \|\mathcal{P}_{\bX}^\perp(\bY)\|^2_2$.}
    \end{align*}

    Next consider the case $\|\beta\|^2_2\leq \frac{4\sqrt{\gamma}}{\lambda_0}$. Then
    \begin{align*}
        &\hat{R}(f,\sample_n;B)
        =\frac{1}{n}\sum_{i=1}^n \min\{(Y_i - X_i^\top \beta)^2, B\}\\
        &\geq \frac{1}{n}\sum_{i=1}^n (Y_i - X_i^\top \beta)^2 - \frac{1}{n}\sum_{i=1}^n (Y_i - X_i^\top \beta)^2 \cdot\one\{(Y_i - X_i^\top \beta)^2\geq B\}\\
        &\geq \frac{1}{n}\|\mathcal{P}_{\bX}^\perp(Y)\|^2_2 - \frac{1}{n}\sum_{i=1}^n (Y_i - X_i^\top \beta)^2 \cdot\one\{(Y_i - X_i^\top \beta)^2\geq B\}\\
        &\geq \frac{1}{n}\|\mathcal{P}_{\bX}^\perp(Y)\|^2_2 - \frac{1}{n}\sum_{i=1}^n \frac{1}{B}(Y_i - X_i^\top \beta)^4\\
        &\geq \frac{1}{n}\|\mathcal{P}_{\bX}^\perp(Y)\|^2_2 - \frac{1}{n}\sum_{i=1}^n \frac{1}{B}\left(8Y_i^4 + 8(X_i^\top\beta)^4\right)\textnormal{\quad since $(a+b)^4\leq 8a^4+8b^4$ for all $a,b\in\R$}\\
        &\geq \frac{1}{n}\|\mathcal{P}_{\bX}^\perp(Y)\|^2_2 - \frac{8(\gamma+\lambda_1 (\frac{4\sqrt{\gamma}}{\lambda_0})^2)}{B}\textnormal{\quad by~\eqref{eqn:concentration_for_lin_extension}}.
    \end{align*}
    Therefore, combining both cases, we have
    \[\hat{R}(\Flin^{(d)},\sample_n;B) 
    \geq \min\left\{ \frac{1}{n}\|\mathcal{P}_{\bX}^\perp(Y)\|^2_2 - \frac{8(\gamma+\lambda_1 (\frac{4\sqrt{\gamma}}{\lambda_0})^2)}{B},\frac{1}{n}\|\mathcal{P}_{\bX}^\perp(\bY)\|^2_2- \frac{4\gamma\lambda_1}{B\lambda_0^2}\right\}.
    \]
    Choosing the constant $c$ appropriately completes the proof.
\end{proof}



\end{document}